\documentclass[reqno,11pt,a4paper]{amsart}

\usepackage{amsmath, amsthm, amsfonts, amssymb, color}

\usepackage{mathtools}
\mathtoolsset{showonlyrefs}

\usepackage[applemac]{inputenc}

\usepackage[all]{xy}

\theoremstyle{plain}
\newtheorem{theorem}{Theorem}
\newtheorem{proposition}[theorem]{Proposition}
\newtheorem{lemma}[theorem]{Lemma}
\newtheorem{corollary}[theorem]{Corollary}

\theoremstyle{definition}

\theoremstyle{remark}
\newtheorem{remark}[theorem]{Remark}

\DeclareMathOperator{\im}{Im}
\def\Z{\mathbb{Z}}	
	
\def\R{\mathbb{R}}	
\renewcommand{\leq}{\leqslant} 		
\renewcommand{\geq}{\geqslant}

\let\ker\relax
\DeclareMathOperator{\ker}{Ker}
\def\cA{\mathcal{A}}
\def\hcA{\hat{\cA}}

\def\hcB{\hat{\mathcal{B}}}

\def\hcC{\hat{\mathcal{C}}}
\def\cF{\mathcal{F}}
\def\hcF{\hat{\cF}}

\def\hcM{\hat{\cM}}

\def\d{\partial}

\def\Dla{D_\lambda}

\begin{document}

\title[Deformations of semisimple Poisson pencils]{Deformations of semisimple Poisson pencils of hydrodynamic type  are unobstructed}

\author{Guido Carlet}

\author{Hessel Posthuma}

\author{Sergey Shadrin}

\address{Korteweg-de Vries Instituut voor Wiskunde, 
Universiteit van Amsterdam, Postbus 94248,
1090GE Amsterdam, Nederland}

\email{g.carlet@uva.nl, h.b.posthuma@uva.nl, s.shadrin@uva.nl}

\begin{abstract}
We prove that the  bihamiltonian cohomology of a semisimple pencil of Poisson brackets of hydrodynamic type vanishes for almost all degrees.
This implies the existence of a full dispersive deformation of a semisimple bihamiltonian structure of hydrodynamic type starting from any infinitesimal deformation. 
\end{abstract}

\maketitle

\tableofcontents

\raggedbottom

\section{Introduction}

\subsection{Basic setup from the theory of integrable hierarchies}

Consider a system of evolutionary PDEs with one spatial variable $x$ and $n$ dependent variables of the form
$$
\frac{\d u^i}{\d t} = A^i_j(u) u^{j}_x + \epsilon \left( B^i_j(u) u^j_{xx} + C^i_{jk} u^j_x u^k_x\right) + O(\epsilon^2).
$$
(Here, and in the following we use the summation convention.)
The right hand side of this equation is represented as a formal power series in $\epsilon$, where the 
coefficient of $\epsilon^k$ is a homogeneous differential polynomial, i.e. an homogeneous polynomial 
in $u^{i,d}:=\d_x^du^i$, $i=1,\dots,n$, $d=1,\dots,k$, $\deg u^{i,d}=d$, of degree $k+1$ whose 
coefficients are smooth functions of coordinates $u^1,\dots,u^n$ on some domain $U\subset 
\mathbb{R}^n$. We can think of the $u^i(x)$ as ($\epsilon$-power series of) smooth functions of $x\in 
S^1$ or Schwarzian functions of $x\in \R$.

On of the possible ways to define and study an integrable hierarchy of partial differential equations of this type uses so-called bihamiltonian structures with hydrodynamic limit. 

A bihamiltonian structure with hydrodynamic limit is given by: a pencil of two compatible Poisson structures on the space of local functionals of  the form
$$
\{u^i(x),u^j(y)\}_a=\left(g^{ij}_a(u)\d_x+\Gamma^{ij}_{k,a}(u) u^k_x \right) \delta(x-y) +O(\epsilon) , \quad a=1,2
$$
where the leading order is given by two Poisson brackets of hydrodynamic type, and the terms of higher order in $\epsilon$ are homogeneous differential operators acting on $\delta(x-y)$, and their coefficients are homogeneous differential polynomials in the dependent variables $u^1, \dots , u^n$;  a system of Hamiltonians of the form
$$
H_a[u]=\int dx\cdot \left( h_a(u)  +O(\epsilon) \right), \quad a=1,2
$$
with higher order terms in $\epsilon$ given by homogeneous differential polynomials in $u^1, \dots  , u^n$.
 
The evolutionary PDE above can be written as an Hamiltonian flow w.r.t. both Poisson structures  
$$
\frac{\d u^i}{\d t} = \{ u^i(x), H_a\}_a  
$$
for $a=1,2$.

The natural equivalence relation on these systems, and in particular on the pencils of Poisson structures with hydrodynamic limit, is given by the so-called Miura transformations, which are transformations of the dependent variables of the form
\begin{equation}
\label{Miura1}
u^i\mapsto v^i(u)+O(\epsilon),
\end{equation}
where higher order terms in $\epsilon$ are homogeneous differential polynomials in $u^1, \dots , u^n$, and 
the leading term is a diffeomorphism. 

%In this context, an important problem is to classify the pencils of Poisson structures with hydrodynamic limit, up to the equivalence given by Miura transformations. 

In this context, an important problem is to classify the Poisson structures with hydrodynamic limit, and respectively the pencils of such Poisson structures, up to the equivalence given by Miura transformations. In the case of a single Poisson structure a triviality theorem~\cite{g02, dz01, dms05, lz11} holds: any Poisson structure with hydrodynamic limit is Miura equivalent to its leading order, i.e., to a Poisson structure of hydrodynamic type. 

The problem of classification of Poisson pencils is more complex. In the scalar $(n=1)$ case a complete solution of this classification problem has been obtained, see~\cite{l02, lz06, al11, lz13, cps14, cps14-2}. 

We are going to consider the general $n>0$ case (see~\cite{b08, dlz06, dz01, lz05, lz11}), where it is convenient to make the assumption that the pencil of Poisson brackets of hydrodynamic type that we are considering is semisimple. A Poisson pencil of hydrodynamic type 
$$
\left(g^{ij}_a(u)\d_x+\Gamma^{ij}_{k,a}(u) u^k_x \right) \delta(x-y),
\quad a=1,2
$$
is semisimple if the polynomial $\det \left(g^{ij}_1-\lambda g_2^{ij}\right)$ of degree $n$ in $\lambda$ has $n$ pairwise distinct non-constant roots on $U\subset \R^n$. In such case~\cite{dlz06} one can use the roots as a set of coordinates on $U$, called canonical coordinates. This choice ensures that both metrics $g_{1,2}^{ij}$ are diagonal with diagonal entries respectively equal to $f^i(u)$, $u^i f^i(u)$, $i=1, \dots , n$ for non-vanishing functions $f^i(u)$ on $U$. 
%
%The Darboux type theorem proved by Getzler~\cite{g02}, ensures that in the case of a single Poisson bracket all higher order terms in $\epsilon$ can be eliminated by a Miura transformation. One can then use a Miura transformation to kill all the higher order terms in $\epsilon$ in the first Poisson bracket. 

In this paper we shall consider the deformation problem of semisimple Poisson pencils of hydrodynamic type
by working in canonical coordinates. The change of coordinates to canonical ones is an example of a 
Miura transformation of the first kind, i.e., a diffeomorphism with the terms $O(\epsilon)$ in \eqref{Miura1} 
equal to zero. By fixing these coordinates, we are therefore left with the problem of classifying Poisson
pencils up to Miura transformations of the second kind, that is, transformations of the dependent 
coordinates as in \eqref{Miura1}, with the zeroth order constant in $\epsilon$ equal to the identity.
Since Miura transformations of the first and second kind obviously generate the whole Miura group, 
this classification problem is equivalent to the original one described above.
Let us now give a precise formulation of this deformation problem in canonical coordinates.

\subsection{Classification of Poisson pencils and the extension problem}
\label{sec:deformation}

Let $\{,\}^0_\lambda= \{,\}^0_2- \lambda \{,\}^0_1$ be a semisimple Poisson pencil of hydrodynamic type~\cite{dlz06}, and let $u^1, \dots , u^n$ be the associated canonical coordinates  over a domain $U \subset \R^n$ where $u^i-u^j\not=0$ for $i\not=j$.
The two compatible Poisson brackets $\{,\}^0_{1,2}$ are of the form
\begin{equation}
\{ u^i(x) , u^j(y) \}^0_a = g^{ij}_a (u(x)) \delta'(x-y) + \Gamma_{k,a}^{ij} (u(x)) u_x^k(x) \delta(x-y), 
\end{equation}
with $a=1,2$, $i, j =1, \dots n$, 
where the contravariant metrics are given by 
\begin{equation}
g_1^{ij} = f^i\delta_{ij} , \quad 
g_2^{ij} = u^i f^i \delta_{ij} \quad\quad \mbox{(no summation over $i$)}
\end{equation} 
and $\Gamma^{ij}_{k,a} = - g_a^{il} \Gamma_{lk,a}^j$, where $\Gamma^j_{lk,a}$ are the Christoffel symbols of the metric $g_a^{ij}$, and $f^1(u), \dots , f^n(u)$ are non-vanishing functions on $U$.

A deformation of $\{,\}^0_\lambda$ is given by a pencil
\begin{equation}
\{, \}_\lambda = \{,\}_2 - \lambda \{, \}_1
\end{equation}
where $\{, \}_a$, $a=1,2$ are compatible Poisson brackets of the form
\begin{equation} \label{deform} 
\{u^i(x), u^j(y) \}_a = \{u^i(x), u^j(y) \}^{0}_a + \sum_{k>0} \epsilon^k
\sum_{l=0}^{k+1} A^{ij}_{k,l;a} (u(x)) \delta^{(l)}(x-y)
\end{equation}
with $A^{ij}_{k,l;a}\in\cA$ and $\deg A^{ij}_{k,l;a} = k-l+1$. Here $\cA$ denotes the space of differential polynomials in $u^1, \dots , u^n$, i.e. formal power series in the variables $u^{i,s}$, $i=1, \dots ,n$, $s>0$, with coefficients that are smooth functions of $u^1, \dots ,u^n$. The degree is defined by setting $\deg u^{i,s} = s$. 

Two deformations are equivalent if they are related by a Miura transformation (of the second kind~\cite{lz11}), i.e. by a change of variables of the form
\begin{equation}
u^i \mapsto \tilde{u}^i = u^i + \sum_{k>0} \epsilon^k F_k^i, \quad i=1, \dots , n
\end{equation}
with $F_k^i \in \cA$ and $\deg F_k^i = k$.

An infinitesimal deformation $\{,\}_\lambda^{\leq2}$ of $\{,\}^0_\lambda$ is given by a pair of 
compatible Poisson brackets of the form~\eqref{deform} where terms of $O(\epsilon^3)$ are 
disregarded. This means that in the expansion \eqref{deform} above, we only consider the 
the coefficients $A^{ij}_{k,l;a}(u(x))$ for $k$ up to $2$, so that the highest derivative $\delta^{(l)}(x-y)$.
appearing is $3$.
Two infinitesimal deformations are equivalent iff they are related by a Miura transformation up to $O(\epsilon^3)$. 
The following theorem, which classifies the deformations of $\{,\}^0_\lambda$, was proved in~\cite{lz05, dlz06}.
\begin{theorem}
Two deformations of $\{,\}^0_\lambda$ are equivalent if and only if the corresponding infinitesimal 
deformations are equivalent. Given an infinitesimal deformation of $\{,\}^0_\lambda$, the functions, called central invariants, defined by
\[
c_i(u):=\frac{1}{3(f^i(u))^2}\left(A^{ii}_{2,3;2}-u^iA_{2,3;1}^{ii}+\sum_{k\not =i}\frac{(A^{ki}_{1,2;2}-u^iA^{ki}_{1,2;1})^2}{f^k(u)(u^k-u^i)}\right),
\]
for $ i=1,\ldots,n,$ only depend on the single variable $u^i$ are invariant under Miura transformations. Two infinitesimal deformations of $\{,\}^0_\lambda$ are equivalent if and only if they have the same central invariants. 
\end{theorem}

The main open problem in the deformation theory of a semisimple Poisson pencil $\{,\}^0_\lambda$ is the problem of extension: making a choice of central invariants $c_1(u^1), \dots , c_n(u^n)$  fixes an equivalence class of infinitesimal deformations of $\{,\}^0_\lambda$, but the question is whether there exists a full deformation $\{,\}_\lambda$ that extends an infinitesimal one to all orders in $\epsilon$. A positive solution of the existence problem was conjectured (and formulated in terms of vanishing of the third bihamiltonian cohomlogy groups) by Liu and Zhang in~\cite{lz05}.

The main result of this paper is the solution of this conjecture, i.e., the affirmative answer to the extension problem.
\begin{theorem} \label{thm:main}
Let $\{,\}_\lambda^{\leq2}$ 
%\begin{equation}
%\{ , \}^{\leq2}_\lambda = \{, \}^0_\lambda + \epsilon \{, \}^1_\lambda + \epsilon^2 \{, \}^2_\lambda
%\end{equation}
be an infinitesimal deformation of a semisimple Poisson pencil of hydrodynamic type $\{, \}^0_\lambda$. Then there exists a deformation $\{ , \}_\lambda$ that extends $\{, \}_\lambda^{\leq2}$ to all orders in $\epsilon$. 
\end{theorem}

\subsection{Methods of proof and organization of the paper}

The problem of the description of automorphisms, infinitesimal deformations, and obstructions to the extension of infinitesimal deformations for any algebraic structure can be formulated in terms of some cohomology groups associated to it. In our case, in order to prove that the deformation of a semisimple pencil of Poisson brackets is not obstructed we have to show that certain cohomology groups, called bihamiltonian cohomology, are equal to zero. 

There is no straightforward way to compute these cohomology groups. However, Liu and Zhang have shown that the vanishing of these cohomology groups follows from the vanishing of the cohomology of the auxiliary complex $(\hcA[\lambda],\Dla)$, defined below, in certain degrees. It is a difficult task to compute the full cohomology of this complex, however we are able to show vanishing of such cohomology in the required degrees using a clever choice of filtrations and the structure of the associated spectral sequences. 

The paper is organized as follows. In Section~\ref{sec:basicformalism} we recall the definition of the auxiliary complex and explain its relation to the problems of deformation of pencils of Poisson structures. In particular, we formulate a statement about the complex $(\hcA[\lambda],\Dla)$ that implies Theorem~\ref{thm:main}. In Section~\ref{sec:filtrations} we introduce a series of filtrations on the complex $(\hcA[\lambda],\Dla)$ that allows us to prove the key statement about its cohomology.

\subsection{Conventions}
Throughout the paper we use the summation convention in the sense that repeated (upper- and 
lower-)indices should be summed over. However, there are a few exceptions when the metrics 
$g_1^{ij}=f^i\delta_{ij}$ and $g_2^{ij}=u^if^i\delta_{ij}$, or the tensors derived from them are involved. 
Such equations always involve the functions $f^i$. To determine which indices are to be summed over,
it suffices to consider the other side of the equation and the indices that appear in there.

\section{Theta formalism, polyvector fields and cohomology}
\label{sec:basicformalism}

The deformation theory of a pencil of Poisson brackets is controlled by the so-called bihamiltonian cohomology defined on the space of local polyvector fields. In order to show the vanishing of such bihamiltonian cohomology in certain degrees, from which Theorem~\ref{thm:main} follows, we consider the cohomology of a related complex $(\hcA[\lambda], \Dla)$, introduced by Liu and Zhang~\cite{lz13}. This approach has a double advantage: first, we can work in the space $\hcA$, where the identifications imposed by integration are not imposed, making computations simpler; second, we can compute the cohomology on the space $\hcA[\lambda]$ of polynomials in $\lambda$, rather than the bihamiltonian cohomology on $\hcA$, and this allows us to use directly the methods of spectral sequences associated with filtrations. 

In this Section we review some basic definitions, mainly from~\cite{lz13}, state our main theorem and derive its most important consequences.

\subsection{Basic definitions} \label{basic-definitions}
Consider the supercommutative associative algebra $\hcA$ defined as  
\begin{equation}
\hcA = C^\infty(U) [[u^{i,1}, u^{i,2}, \dots; \theta_i^0, \theta_i^1, \theta_i^2, \dots ]],
\end{equation}
where $u^{i,s}$, $i=1,\dots,n$, $s=1,2,\dots$ are  formal even variables and $\theta_i^s$, $i=1,\dots,n$, $s=0,1,2,\dots$ are  odd variables.  An element in $C^\infty (U)$ is represented by a function of the coordinates $u^i$, $i=1,\dots,n$ on the domain $U \subset \R^n$.
We define the \emph{standard gradation} on $\hcA$ by assigning the degrees 
\begin{equation}
\deg u^{i,s} = \deg \theta^s_i = s, \quad s=1,2,\dots
\end{equation}
and degree zero to both $\theta_i=\theta_i^0$ and the elements in $C^\infty (U)$. The standard degree $d$ homogeneous component of $\hcA$ is denoted $\hcA_d$. 
Notice that $\hcA_d$ coincides with the standard degree $d$ homogeneous component of the polynomial algebra $C^\infty(U) [u^{i,1},  \dots , u^{i,d}; \theta_i^0, \dots , \theta_i^d ]$.
 The {\it super gradation}, that we will denote by $\deg_\theta$, is defined by assigning degree one to $\theta_{i}^{s}$ for $s\geq0$ and degree zero to the remaining generators of $\hcA$. The super degree $p$ homogeneous component is denoted $\hcA^p$. We also denote 
\begin{equation}
\hcA_d^p=\hcA_d \cap \hcA^p .
\end{equation}
The standard derivation on $\hcA$,
\begin{equation}
\partial = \sum_{s\geq0} \left( u^{i,s+1} \frac{\partial }{\partial u^{i,s}} + \theta_i^{s+1} \frac{\partial }{\partial \theta_i^s} \right),
\end{equation}
is compatible with the standard and super gradations, in particular it increases the standard degree by one and leaves invariant the super degree. 
Thanks to the homogeneity of $\partial$, the space $\hcF := \frac{\hcA}{\partial \hcA}$ still possesses  two gradations, which we keep denoting with indices $p$ and $d$. The elements of $\hcF^p$ are called {\it local $p$-vectors} and the projection map is denoted by an integral
\begin{equation}
\int: \hcA \to \hcF.
\end{equation}
The space $\hcF$ can be endowed with the Schouten-Nijenhuis bracket
\begin{equation}
[,]:\hcF^p \times \hcF^q \to \hcF^{p+q-1} ,
\end{equation}
which satisfies the usual graded skew-symmetry and graded Jacobi identities, see~\cite{lz13, lz11} for  more details.

A Poisson bivector $P$ is an element of $\hcF^2$ that satisfies $[P,P]=0$. If $P$ is a Poisson bivector, its adjoint action $d_P = [P, \cdot ]$ on $\hcF$ by the graded Jacobi identity squares to zero, hence defines a differential complex $(\hcF, d_P)$. 
Given a Poisson bivector $P$, the super derivation $D_P$ on $\hcA$ is defined by
\begin{equation}
D_P = \sum_{s\geq0}  \left( \partial^s \left( \frac{\delta P}{\delta \theta_i}  \right) \frac{\partial }{\partial u^{i,s}} + \partial^s  \left( \frac{\delta P}{\delta u^i}  \right) \frac{\partial }{\partial \theta_i^s} \right),
\end{equation}
where the variational derivatives on $\hcA$ w.r.t. $u^i$ and $\theta_i$ are defined as follows 
\begin{equation}
\frac{\delta}{\delta u^i} =\sum_{s=0}^\infty (-\partial)^s \frac{\partial}{\partial u^{i,s}}, 
\qquad
\frac{\delta}{\delta \theta_i}=\sum_{s=0}^\infty (-\partial)^s \frac{\partial}{\partial \theta_i^s}.
\end{equation}
The super derivation $D_P$ squares to zero, and is such that the integral defines a map of differential complexes
\begin{equation}
\int : (\hcA, D_P) \to (\hcF, d_P) .
\end{equation}
As pointed out in~\cite{lz13} this allows us to work with the complex $(\hcA, D_P)$ rather than with the more complicated $(\hcF, d_P)$. 

A Poisson pencil is given by two Poisson bivectors $P_1$, $P_2$ which are compatible, i.e. $[P_1, P_2]=0$. For each $\lambda$, then, $P_\lambda : =P_2 - \lambda P_1$ is also a Poisson bivector. We denote by $d_1$ and $d_2$ the differentials on $\hcF$ corresponding to $P_1$ and $P_2$, respectively. Due to compatibility, $d_\lambda := d_2 - \lambda d_1$ squares to zero. We denote by $D_1$ and $D_2$ the super derivations on $\hcA$ associated to $P_1$ and $P_2$, respectively. Their compatibility in this case implies that $D_\lambda := D_2 - \lambda D_1$ also squares to zero. In summary we can define two differential complexes associated to a Poisson pencil
\begin{equation}
(\hcA[\lambda], D_\lambda ) , \quad  (\hcF[\lambda], d_\lambda) .
\end{equation}

\begin{remark}
The Poisson brackets $\{,\}_{1,2}$ introduced in~\eqref{deform} are elements of the space $\Lambda_{loc}^2$ of local bivectors written in $\delta$-formalism, see~\cite{dz01}. There is a one to one correspondence between the space of local $p$-vectors $\Lambda^p_{loc}$, written in $\delta$-formalism and the space $\hcF^p$. We will not recall it here in general but rather refer the reader to~\cite{lz13}. For the case of a bivector, written as
\begin{equation}
\{u^i(x), u^j(y) \}  = \sum_{s\geq0} B_s^{ij} \delta^{(s)}(x-y),
\end{equation}
with $B_s^{ji} \in \cA$, 
the corresponding element in $\hcF^2$ is 
\begin{equation}
P = \frac12 \int \theta_i  \sum_{s\geq0} B_s^{ij}  \theta^s_j .
\end{equation}

\end{remark}

\subsection{The complex $(\hcA[\lambda], D_\lambda)$}
Let us fix a semisimple Poisson pencil of hydrodynamic type  $\{,\}^0_\lambda= \{,\}^0_2- \lambda \{,\}^0_1$, and denote by $u^1, \dots , u^n$  the associated canonical coordinates on $U\subset\R^n$, see \S\ref{sec:deformation}. The compatible Poisson brackets $\{,\}^0_{1,2}$ are represented in $\hcF^2_1$ by two bivectors 
\begin{equation}
P_a = \frac12 \int \left( g^{ij}_a \theta^0_i \theta^1_j + \Gamma_{k,a}^{ij} u^{k,1} \theta_i \theta_j  \right), \quad a=1,2 .
\end{equation}
In canonical coordinates the Christoffel symbol of $g_1^{ij} = f^i \delta_{ij}$ is
\begin{equation}
\Gamma_{k,1}^{ij}  = 
\frac12 \left(  \partial_k f^i  \delta_{ij} 
+\frac{f^i}{f^j} \partial_i f^j \delta_{jk} 
-\frac{f^j}{f^i} \partial_j f^i \delta_{ik}
\right)
\end{equation}
and $P_1$ is given by
\begin{equation}
P_1 = \frac12 \int \left( f^i \theta_i \theta_i^1 +  \frac{f^i}{f^j} \partial_i f^j u^{j,1} \theta_i \theta_j \right) .
\end{equation}
We denote by $D_1=D(f^1,\dots,f^n) $ the super derivation on $\hcA$  corresponding to $P_1$. A straightforward computation gives us the following formula
% checked first 3 lines 8/15 gc
\begin{align}
\label{diff-biham}
& D(f^1,\dots,f^n) = \sum_{s\geq0}  \partial^s\left(f^i\theta_i^1\right)\frac{\d}{\d u^{i,s}}
\\ \notag 
& +\frac{1}{2}\sum_{s\geq0}  \partial^s\left(
\d_jf^i u^{j,1} \theta_i^0 
+ f^i\frac{\d_i f^j}{f^j} u^{j,1}\theta_j^0 
- f^j\frac{\d_j f^i}{f^i} u^{i,1}\theta_j^0
\right) \frac{\d}{\d u^{i,s}}
\\ \notag
& + \frac{1}{2} \sum_{s\geq0}  \partial^s \left( 
\d_if^j  \theta_j^0 \theta_j^1 
+ f^j\frac{\d_j f^i}{f^i} \theta_i^0 \theta_j^1 
- f^j\frac{\d_j f^i}{f^i} \theta_j^0 \theta_i^1
\right) \frac{\d}{\d \theta_i^s}
\\ \notag
%last two lines :
&+\frac12 \sum_{s\geq0} \partial^s \Big( 
\frac{\partial }{\partial u^i} \big( \frac{f^k}{f^j}\partial_k f^j \big) u^{j,1} \theta_k \theta_j  
-\frac{\partial }{\partial u^j} \big( \frac{f^k}{f^i}\partial_k f^i \big) u^{j,1} \theta_k \theta_i \Big) \frac{\partial }{\partial \theta_i^s} .
% alternative version of last two lines:
%& +
%\frac12 \sum_{s\geq 0} \d^s\left(
%f^j \frac{\d_i f^l}{f^l} \frac{\d_i f^l}{f^l} u^{l,1} \theta_l^0 \theta_j^0
%- f^l \frac{\d_i f^l}{f^l} \frac{\d_l f^j}{f^j} u^{j,1} \theta_l^0 \theta_j^0
%\right. 
%\\ \notag
%& + f^l \frac{\d_l f^i}{f^i} \frac{\d_i f^j}{f^j} u^{j,1} \theta_l^0 \theta_j^0
%- \frac{f^lf^j}{f^i} \frac{\d_l f^i}{f^i} \frac{\d_j f^i}{f^i} u^{i,1} \theta_l^0 \theta_j^0
%\\ \notag
%& 
%+ f^l \frac{\d_l f^i}{f^i} \frac{\d_l f^j}{f^j} u^{j,1} \theta_j^0 \theta_i^0
%- f^j \frac{\d_l f^i}{f^i} \frac{\d_j f^l}{f^l} u^{l,1} \theta_j^0 \theta_i^0
%\\ \notag &
%\left.
%+f^l \frac{\d_lf^i}{f^i} \frac{\d_j f^l}{f^l} u^{j,1} \theta_l^0 \theta_i^0
%+f^l \frac{\d_lf^i}{f^i} \frac{\d_l f^j}{f^j} u^{j,1} \theta_l^0 \theta_i^0
%\right) \frac{\d}{\d\theta^s_i}.
\end{align}

Notice that $D$ is an homogeneous operator of standard degree one, therefore it is well-defined on $\hcA[\lambda]$, because each homogeneous component $\hcA_d[\lambda]$ the infinite sums appearing in~\eqref{diff-biham} have only a finite number of non-vanishing terms.

The super derivation corresponding to $P_2$ is then given by 
\begin{equation}
D_2:=D(u^1f^1,\dots,u^nf^n) .
\end{equation} 
Our aim is to compute the cohomology of the complex $(\hcA[\lambda], D_\lambda)$ with $D_\lambda = D_2 - \lambda D_1$, and  $D_1$, $D_2$ given above.

%the Christoffel symbols 
%\begin{equation}
%\Gamma_{k,2}^{ij}  = 
%\frac12 \left( f^i \delta_{ij} \delta_{ik} +
%u^i \partial_k f^i  \delta_{ij} 
%+u^i \frac{f^i}{f^j} \partial_i f^j \delta_{jk} 
%-u^j\frac{f^j}{f^i} \partial_j f^i \delta_{ik}
%\right)
%\end{equation}

\subsection{The main theorem}
In Section 3 we prove the following vanishing theorem for the cohomology of the complex $(\hcA[\lambda], D_\lambda)$. 
\begin{theorem} 
\label{thm:A}
The cohomology $H^p_d(\hcA[\lambda],D_\lambda)$ vanishes for all bi-degrees $(p,d)$, unless
\begin{equation}
d=0, \dots, n, \quad p=d, \dots , d+n, 
\end{equation}
or
\begin{equation}
d=n+1,n+2, \quad p=d, \dots , d+n-1.
\end{equation}
\end{theorem}

\begin{remark}
For $n=1$ the bi-degrees for which we cannot state vanishing according to Theorem~\ref{thm:A} are 
\begin{equation}
(p,d) = (0,0), (1,0), (1,1), (2,1), (2,2), (3,3).
\end{equation}
In~\cite{cps14-2} (and in~\cite{cps14} for the KdV case) we compute by a different method the cohomology $H^p_d(\hcA[\lambda],D_\lambda)$ for all bi-degrees, proving in particular that it vanishes also for 
\begin{equation}
(p,d) = (1,0), (1,1), (2,2). 
\end{equation}
This shows that the vanishing theorem above can be strengthened. Our result is enough, however, to prove the absence of obstructions to deformations of bihamiltonian structures, which is our main aim.
\end{remark}

%\begin{remark}
%For $n=2$ the cohomology vanishes for all $(p,d)$, but for 13 cases. For arbitrary $n$ the cohomology vanishes for all bi-degrees, but for $(n+2)^2 -3$ cases.
%\end{remark}

\begin{remark}
\label{rmk:cases}
In the proof of Theorem~\ref{thm:A} we will distinguish two sets of indexes for which we do not prove vanishing:
\begin{equation}
d=0, \dots ,n, \quad p=d, \dots , d+n \quad \text{(Case 1),}
\end{equation}
and 
\begin{equation}
d=2, \dots , n+2, \quad p=d, \dots ,d+n-1 \quad \text{(Case 2),}
\end{equation}
which clearly overlap for $n\geq2$. We distinguish the two cases since they have different sources in the complex.
\end{remark}

\subsection{Vanishing of bihamiltonian cohomology}

By the following lemma of Liu and Zhang~\cite{lz13}, the bihamiltonian cohomology of $\hcF$ is isomorphic to the cohomology of the complex\footnote{A similar description in terms of a bicomplex was given in~\cite{b08}.} $(\hcF[\lambda],d_\lambda)$ in almost all degrees $(p,d)$.  
Let $(C, \partial_1, \partial_2)$ be either the double complex $(\hcA, D_1, D_2)$ or $(\hcF, d_1, d_2)$. The bihamiltonian cohomology of the double complex $(C, \partial_1,\partial_2)$ is defined by
\begin{equation}
BH(C, \partial_1, \partial_2) = \frac{\ker \partial_1 \cap \ker \partial_2}{\im \partial_1 \partial_2} .
\end{equation}

\begin{lemma}%[\cite{lz13}] 
\label{blzlemma}
The natural embedding of $C$ in $C[\lambda]$ induces an isomorphism
\begin{equation}
BH^p_d(C,\partial_1, \partial_2 )  \cong H^p_d(C[\lambda],\partial_\lambda) 
\end{equation}
for $p\geq0$, $d\geq2$.
\end{lemma}

%
%\begin{remark}
%THE POINT OF THIS REMARK IS TO SEE IF THE PREVIOUS LEMMA CAN BE GENERALIZED, AND GIVE DETAILS OF THE PROOF.
%
%Observe that for pairs $(p,d)$ such that $p=0$ or $d=0$, and $H^p_d(C,\partial_1) =0$, both $BH^p_d(C, \partial_1,\partial_2)$ and $H^p_d(C[\lambda],\partial_\lambda)$ vanish. 
%% since in this case the cohomology coincides with the kernel, which is zero. 
%
%SO THE ISOM IS THERE FOR $p=0$. WHAT HAPPENS FOR $p=1$ ?
%
%For $p,d\not= 0$ the natural embedding of $C$ in $C[\lambda]$, that assigns to $c\in C$ the corresponding degree zero polynomial in $C[\lambda]$, induces a map 
%\begin{equation}
%C^p_d\cap \ker \partial_1 \cap \ker \partial_2 \to H^p_d(C[\lambda],\partial_\lambda)
%\end{equation}
%which is surjective if $H^p_d(C,\partial_1)=0$.
%\end{remark}

As pointed out in~\cite{lz13}, the short exact sequence of complexes
\begin{equation}
0 \to (\hcA[\lambda]/\R[\lambda], D_\lambda) \to (\hcA[\lambda],D_\lambda) \to (\hcF[\lambda], d_\lambda) \to 0
\end{equation}
implies a long exact sequence in cohomology, which includes 
\begin{equation}
\cdots \to H^p_d(\hcA[\lambda]) \to H^p_d(\hcF[\lambda]) \to 
H^{p+1}_d(\hcA[\lambda]) \to \cdots
\end{equation}
for $p, d\geq0$.
It is clear that if both $H^p_d(\hcA[\lambda], D_\lambda)$ and $H^{p+1}_d(\hcA[\lambda], D_\lambda)$ vanish, then $H_d^p(\hcF[\lambda],d_\lambda)$ vanishes too. 
Our vanishing result for $H^p_d(\hcA[\lambda])$ translates to the following statement for the cohomology of the complex $(\hcF[\lambda], d_\lambda)$.
\begin{corollary}
The cohomology $H^p_d(\hcF[\lambda], d_\lambda)$ vanishes for all bi-degrees $(p,d)$, unless
\begin{equation}
d=0, \dots,n, \quad p=d-1, \dots , d+n,
\end{equation}
or
\begin{equation}
d=n+1, n+2, \quad p=d-1, \dots , d+n-1.
\end{equation}
\end{corollary}

Using the isomorphism of Lemma~\ref{blzlemma} we obtain the vanishing of the bihamiltonian cohomology of $\hcF$.
\begin{corollary}
The bihamiltonian cohomology $BH^p_d(\hcF, d_1, d_2)$ vanishes for all bi-degrees $(p,d)$ with $d\geq2$, unless
\begin{equation}
 d=2, \dots,n, \quad p=d-1, \dots , d+n,
\end{equation}
or
\begin{equation}
d=n+1, n+2, \quad p=d-1, \dots , d+n-1.
\end{equation}
\end{corollary}

\begin{remark}
Notice that in particular it follows that
\begin{equation}
BH^2_{\geq4}(\hcF, d_1,d_2) = 0,\qquad 
BH^3_{\geq5}(\hcF, d_1, d_2) = 0 .
\end{equation}
The vanishing of the second cohomology for $d\geq4$ has been already proved in~\cite{lz05, dlz06}, together with the results
\begin{equation}
 BH^2_{2}(\hcF, d_1,d_2) = 0, \qquad  
 BH^2_{3}(\hcF, d_1,d_2) \cong \bigoplus_{i=1}^n C^\infty (\R) .
\end{equation}
The vanishing of the third cohomology for $d\geq5$ is new, and is the most relevant for the  extension problem of deformation theory.
\end{remark}

\subsection{Bihamiltonian cohomology and deformations}
The deformation problem can be formulated in $\hcF$ as follows. 
Let $P^0_\lambda=P^0_2 - \lambda P^0_1 \in \hcF^2_{1}[\lambda]$ be a semisimple Poisson pencil of hydrodynamic type. 
A deformation of $P^0_\lambda$ is given by $P_\lambda = P_2- \lambda P_1 \in \hcF_{\geq1}^2[\lambda]$ with $P_\lambda - P_\lambda^0 \in \hcF^2_{\geq2}$ and $[P_\lambda,P_\lambda]=0$.
%Two deformations $P_\lambda$, $Q_\lambda$ of $P^0_\lambda$ are equivalent if $P_\lambda = e^{\mathrm{ad}_X} Q_\lambda$ for $X\in \hcF^1_{\geq1}$. 
An infinitesimal deformation of $P^0_\lambda$ is given by 
\begin{equation}
P_\lambda^{\leq2} = P^0_\lambda + P^1_\lambda +P^2_\lambda, \quad P_\lambda^{d-1} = P_2^{d-1} - \lambda P_1^{d-1} \in \hcF^2_d [\lambda] 
\end{equation}
such that 
\begin{equation}
[ P_\lambda^{\leq2} , P_\lambda^{\leq2}] \in \hcF^3_{\geq5}[\lambda] .
\end{equation}
A deformation $P_\lambda$ of $P_\lambda^0$ extends a infinitesimal deformation $P_\lambda^{\leq2}$ if $P_\lambda - P^{\leq2}_\lambda \in \hcF^3_{\geq4}[\lambda]$.

The extension problem can be stated as follows: 
{\it given an infinitesimal deformation $P_\lambda^{\leq2}$ of $P_\lambda^0$, there exists a deformation $P_\lambda$ extending $P_\lambda^{\leq2}$ ?}

As pointed out in~\cite{lz13} the fact that the second bihamiltonian cohomology groups $BH^2_d(\hcF, d_1, d_2)$ vanish for $d\geq4$ and $d=2$ but not for $d=3$, implies that the vanishing of $BH^3_{\geq6}(\hcF,d_1,d_2)$ guarantees that any infinitesimal deformation $P_\lambda^{\leq2}$ can be extended to a full deformation $P_\lambda$. 

This result, expressed in the $\delta$-formulation of local bivectors, is Theorem~\ref{thm:main}. 

To see that the vanishing of $BH^3_{\geq 6}(\hcF)$ implies that the deformations are unobstructed, consider that any infinitesimal deformation, thanks to the triviality theorem for the deformations of a single Poisson structure~\cite{g02, dz01, dms05, lz11}, can be put in the form
\begin{equation}
P_1^{\leq 2} = P_1^0, \quad P_2^{\leq 2} = P_2^0 + P_2^2 
\end{equation}
by an appropriate Miura transformation. Clearly $P_2^2$ is in the kernel of both $d_1$ and $d_2$, hence identifies an element of $BH^2_3(\hcF)$. Let us look for a deformation of the form
\begin{equation} \label{evendef}
P_1 = P_1^0 , \quad P_2 = P_2^0 + P_2^2+P_2^4 +P_2^6 + \dots 
\end{equation}
Let us first show that exist a term $P_2^4$ such that 
\begin{equation} \label{defeq1}
d_1 P_2^4 = 0, \quad  d_2 P_2^4 +\frac12 [ P_2^2 , P_2^2 ] =0 .
\end{equation}
Clearly $[P_2^2,P_2^2] \in \hcF^3_6$ is in $\ker d_1 \cap \ker d_2$, hence the vanishing of $BH^3_6(\hcF)$ implies that 
\begin{equation} 
d_1 d_2 Q_1 = \frac12 [P_2^2,P_2^2] 
\end{equation}
for some $Q_1 \in \hcF^1_4$. Then  $P_2^4 = d_1 Q_1$ gives a solution of~\eqref{defeq1}. 

At the next step of the deformation we want to find $P_2^6$ such that
\begin{equation} \label{defeq2}
d_1 P_2^6 =0, \quad d_2 P_2^6 + [P_2^2, P_2^4] =0 .
\end{equation}
As before $[P_2^2,P_2^4] \in \hcF^3_8$ is in $\ker d_1 \cap \ker d_2$, hence the vanishing of $BH^3_8(\hcF)$ implies that there is an element $Q_2\in\hcF^1_6$ s.t. $d_1 d_2 Q_2 = [P_2^2,P_2^4]$. Then setting $P_2^6 = d_1 Q_2$ gives us the required solution of~\eqref{defeq2}.

Since the vanishing of  $BH_{\geq6}^3(\hcF)$ ensures that this procedure can be continued indefinitely, as proved by induction in~\cite{lz13}, the existence of a full deformation extending the infinitesimal deformation $P_\lambda^{\leq2}$ indeed follows. 

Notice that this in particular implies that any deformation can be put in the form~\eqref{evendef}, i.e., with the first Poisson tensor undeformed and the second one having only odd standard degree terms. This fact was also proved independently of the vanishing of the third bihamiltonian cohomology in~\cite{dlz06}.

\section{Filtrations and spectral sequences}
\label{sec:filtrations}

In this Section we give a proof of Theorem~\ref{thm:A}. It is a rather technical argument: basically, we introduce a sequence of filtrations  and study the associated spectral sequences in order to show the vanishing of the cohomology in some degrees. 

\subsection{Strategy of the proof of Theorem~\ref{thm:A}}
\label{strategy}
Before we start with the technical details, let us make some general remarks about the strategy of the proof. 
The theorem states that the 
cohomology of the complex $(\hcA[\lambda],\Dla)$ vanishes in certain degrees $(p,d)$. 
We prove this vanishing by introducing several spectral sequences associated to certain filtrations. 

Central to this derivation is the following simple principle: suppose we have a cochain complex $(C,d)$
with a bounded decreasing filtration 
\[
\ldots \subset F^{p+1}C\subset F^pC\subset\ldots
\]
The associated spectral sequence is bounded and converges to 
\[
E_1^{p,q}=H^{p+q}(F^pC\slash F^{p+1}C,d_0)\Longrightarrow H^{p+q}(C,d),
\]
where $d_0$ is the induced differential on the zeroth page, i.e. the associated graded complex.
%This means that 
%\[
%E_\infty^{p,q} \cong \frac{F^p H^{p+q}(C,d)}{F^{p+1} H^{p+q}(C,d)}
%\]
%where $E_\infty^{p,q}$ is the stable value and $F^p H(C,d)$ is the bounded filtration induced on the cohomology.
Suppose now that $H^p(E_k,d_k)=0$ for some $k\geq 0$.
Since all higher pages are iterated subquotients of $E_k$, we see that in this case $H^p(C,d)=0$.

In the proof we will apply this principle inductively: we introduce a filtration on the complex 
$(\hcA[\lambda],\Dla)$ which induces a spectral sequence. To show the vanishing of the cohomology of the 
first page $E_1$, we introduce another filtration on this page, which induces a spectral sequence
converging to the second page $E_2$ of the previous spectral sequence. In this way, we apply in total a sequence of three filtrations. For the convenience of the reader, we list them below:
\begin{itemize}
\item[1.] The first filtration is associated with the degree of monomials in $\hcA[\lambda]$ in the variables 
$u^{i,s}$, $s\geq 1$, i.e. we assign degree $1$ to each $u^{i,s}$, $s\geq 1$. The differential  on the zeroth page of the associated spectral sequence is the part of $\Dla$ that preserves this degree, whereas on the first page it is the part that decreases it by $1$.
\item[2.] On the first page of the spectral sequence above, we consider the filtration given by the degree in $\theta_i^1$, for all $i=1,\ldots,n$, i.e., this time $\deg(\theta_i^1)=1,~i=1,\ldots,n$. On the zeroth page, the differential is given by the part of the differential in 1 (at the first page) that increases the number of $\theta^1_i$ by $1$.
\item[3.] The complex on the zeroth page in 2 splits as direct sum of complexes $\hcC_i,~i=1,\ldots,n$.
To compute the cohomology of the subcomplex $\hcC_i$, we filter by the degree of monomials in $\theta_i^1$, where this time $i$ is fixed. On the zeroth page of the spectral sequence we finally find a complex of which we can prove the vanishing of the cohomology in the relevant degrees.
\end{itemize}
Having obtained the vanishing of the cohomology in 3, we apply the principle stated above to argue that the first page of the spectral sequence in 2 vanishes in certain degree. The same argument gives  the vanishing of the second page of the spectral sequence in 1, which in turn proves the vanishing of the cohomology of the original complex $(\hcA[\lambda],\Dla)$ in certain degrees, i.e. Theorem \ref{thm:A}.

\subsection{Grading and subcomplexes}

Since $\theta^s_i$  are odd variables, we have a restriction on possible gradings $p$ and $d$ on the space $\hcA$ . Indeed, the minimal possible  standard degree $d$ of a monomial in $\theta^s_i$ of super degree $p=nq+r$, $r<n$ is the degree of $\theta_1^0\cdots\theta_n^0\cdots\theta_1^{q-1}\cdots\theta_n^{q-1} \theta_{i_1}^q \cdots \theta_{i_r}^q$ equal to $n(0+\cdots+(q-1))+rq=nq(q-1)/2+rq$. So, the finitely generated $C^\infty (U)$-module $\hcA^p_d$ is zero for 
\begin{equation} \label{ineqqq}
d<nq(q-1)/2+rq.
\end{equation}
Moreover, the operators $D_1$ and $D_2$ (and, correspondingly, $d_1$ and $d_2$) are of bi-degree $(p,d)$ equal to  $(1,1)$. Therefore, the difference $p-d$ is preserved by both operators, and this means that the  space $\hcA$ can be seen as the completion of an infinite direct sum of subcomplexes indexed by difference $d-p=-n,-n+1,\dots$. The inequality~\eqref{ineqqq} implies that each of these subcomplexes is finite, and consequently we can compute the cohomology of each of them separately.
For this general reason all the filtrations and spectral sequences introduced below will be bounded, and consequently the spectral sequences will converge in a finite number steps.

\subsection{The first filtration}

Let us define a degree $\deg_u$ on $\hcA[\lambda]$ by assigning degree one to $u^{i,s}$ with $s\geq 1$ and degree zero to the remaining generators. 

The differential $D_{\lambda}$ splits in homogeneous components $\Delta_k$ with respect to the degree $\deg_u$, i.e., 
\begin{equation}
\Dla= \Delta_{-1}+\Delta_0+\dots,
\end{equation}
where $\deg_u \Delta_k = k$. 

One can easily check from~\eqref{diff-biham} that the only term that lowers the degree $\deg_u$ is 
%where where $\Delta_{-1}$ decreases $q$ by 1, $\Delta_0$ preserves $q$, and $\Delta_t$, $t\geq 1$, increases $q$ by $t$. Explicitly, we see from equation \eqref{diff-biham}, that there is only one term that lowers degree $q$, so that we have:
\begin{equation}
\label{formula-0}
\Delta_{-1}  = \sum_{s\geq 1}  (-\lambda+u^i)f^i\theta_i^{s+1}\frac{\d}{\d u^{i,s}}.
\end{equation}

The terms in \eqref{diff-biham} that preserve the degree $\deg_u$ are:
% re-checked this formula on 24/7 gc
\begin{align}
\Delta_0 & = 
(-\lambda+u^i)f^i\theta_i^{1}\frac{\d}{\d u^i} 
\\ \notag &
+
\sum_{\substack{
s=a+b \\
s, a \geq 1;
b\geq 0
}} (-\lambda+u^i) \binom {s}{b} \partial_j f^i u^{j,a} \theta_i^{1+b} \frac{\d}{\d u^{i,s}}
+ \sum_{\substack{
s=a+b \\
s, a \geq 1;
b\geq 0
}} \binom {s}{b} f^i u^{i,a} \theta_i^{1+b} \frac{\d}{\d u^{i,s}}
\\ \notag
& +\frac{1}{2}\sum_{\substack{
s=a+b \\
s \geq 1;
a,b\geq 0
}}  (-\lambda+u^i) \binom{s}{b}
\d_jf^i u^{j,1+a} \theta_i^b
\frac{\d}{\d u^{i,s}}
+\frac{1}{2}\sum_{\substack{
s=a+b \\
s \geq 1;
a,b\geq 0
}}  \binom{s}{b}
f^i u^{i,1+a} \theta_i^b
\frac{\d}{\d u^{i,s}}
\\ \notag
& +\frac{1}{2}\sum_{\substack{
s=a+b \\
s \geq 1;
a,b\geq 0
}}  (-\lambda+u^i) \binom{s}{b}  
f^i\frac{\d_i f^j}{f^j} u^{j,1+a}\theta_j^b 
\frac{\d}{\d u^{i,s}}
+\frac{1}{2}\sum_{\substack{
s=a+b \\
s \geq 1;
a,b\geq 0
}}  \binom{s}{b}  
f^i u^{i,1+a}\theta_i^b 
\frac{\d}{\d u^{i,s}}
\\ \notag
& -\frac{1}{2}\sum_{\substack{
s=a+b \\
s \geq 1;
a,b\geq 0
}}
 (-\lambda+u^j) \binom{s}{b} 
f^j\frac{\d_j f^i}{f^i} u^{i,1+a}\theta_j^b
\frac{\d}{\d u^{i,s}}
-\frac{1}{2}\sum_{\substack{
s=a+b \\
s \geq 1;
a,b\geq 0
}}
\binom{s}{b} 
f^i u^{i,1+a}\theta_i^b
\frac{\d}{\d u^{i,s}}
\\ \notag
& + \frac{1}{2} \sum_{\substack{
s=a+b \\
s,a,b\geq 0
}}  
(-\lambda+u^j)\binom{s}{b}
\d_if^j  \theta_j^a \theta_j^{1+b} 
 \frac{\d}{\d \theta_i^s}
+ \frac{1}{2} \sum_{\substack{
s=a+b \\
s,a,b\geq 0
}}  
\binom{s}{b}
f^i  \theta_i^a \theta_i^{1+b} 
 \frac{\d}{\d \theta_i^s}
\\ \notag
&
+ \frac{1}{2} \sum_{\substack{
s=a+b \\
s,a,b\geq 0
}}  
(-\lambda+u^j)\binom{s}{b}
f^j\frac{\d_j f^i}{f^i} \theta_i^a \theta_j^{1+b}   \frac{\d}{\d \theta_i^s}
+\frac{1}{2} \sum_{\substack{
s=a+b \\
s,a,b\geq 0
}}  
\binom{s}{b}
f^i \theta_i^a \theta_i^{1+b}   \frac{\d}{\d \theta_i^s}
\\ \notag
&
- 
\frac{1}{2} \sum_{\substack{
s=a+b \\
s,a,b\geq 0
}}  
(-\lambda+u^j)\binom{s}{b}
f^j\frac{\d_j f^i}{f^i} \theta_j^a \theta_i^{1+b}
 \frac{\d}{\d \theta_i^s}
- 
\frac{1}{2} \sum_{\substack{
s=a+b \\
s,a,b\geq 0
}}  
\binom{s}{b}
f^i \theta_i^a \theta_i^{1+b}
 \frac{\d}{\d \theta_i^s} .
\end{align}

Suppose we introduce a filtration of $\hcA[\lambda]$ related to the degree $\deg_u$.
The usual decreasing filtration associated with $\deg_u$ is defined as follows: let $\tilde{F}^r \hcA[\lambda]$ be the subspace of $\hcA[\lambda]$ of elements with homogenous components of $\deg_u$ greater of or equal to $r$. This filtration however is not preserved by $D_\lambda$, because of the presence of the term of negative degree $\Delta_{-1}$. 

Let us instead consider the decreasing filtration $F \hcA[\lambda]$ of $\hcA[\lambda]$ associated with the degree $\deg_u + \deg_\theta$. Let $F^r \hcA[\lambda]$ be given by the elements in $\hcA[\lambda]$ with homogeneous components in $\deg_u + \deg_\theta$ of degree $\geq r$, where $\deg	_\theta$ is  the super gradation defined in \S~\ref{basic-definitions}. We get
\begin{equation}
\label{flit-complex}
\dots\subset F^{2}\hcA[\lambda] \subset F^{1}\hcA[\lambda] \subset F^{0}\hcA[\lambda] = \hcA[\lambda].
\end{equation}
Since $\deg_\theta D_\lambda = 1$, the differential $D_\lambda$ preserves this filtration. Notice that this filtration, when restricted to each subcomplex $\hcA^p_d[\lambda]$ with fixed difference $d-p$, is bounded.

Let us now consider the spectral sequence associated with the filtration $F \hcA[\lambda]$. Since such filtration comes from a grading, the zeroth page is given by
\[
E_0 = \bigoplus_{p,q}E_0^{p,q} = \bigoplus_{p,q} F^p\hcA^{p+q}[\lambda]\slash F^{p+1}\hcA^{p+q}[\lambda]\cong \hcA[\lambda] 
\]
and the induced differential by $\Delta_{-1}$, i.e., 
\begin{equation}
(E_0,d_0)=(\hcA[\lambda],\Delta_{-1}) .
\end{equation}
To obtain the first page of the spectral sequence we need therefore to compute the cohomology of the complex $( \hcA[\lambda], \Delta_{-1} )$. 

We define
\begin{equation}
\hcC:=C^\infty(U)[[\theta^0_1,\dots,\theta^0_n,\theta^1_1,\dots,\theta^1_n]],
\end{equation}
and
\begin{equation}
\hcC_i := \hcC[[ \{ u^{i,s}, \theta_i^{s+1}, s\geq1 \} ]]. 
\end{equation}
On $\hcC_i$, we denote by $\hat{d}_i$ the de Rham differential
\[
\hat{d}_i = \sum_{s\geq1} \theta_i^{s+1} \frac{\partial }{\partial u^{i,s}}  .
\]
% checked and corrected from here 30/3:
\begin{proposition} \label{Hd0}
The cohomology of $\Delta_{-1}$ is given by
\[
H (\hcA[\lambda] , \Delta_{-1}) \cong \hcC[\lambda] \oplus \bigoplus_{i=1}^n \im  \left( \hat{d}_i : \hcC_i \to \hcC_i \right) .
\]
\end{proposition}
Before we start the proof of this proposition, as a preliminary step, we observe that the cohomology of the de Rham complex $(\hcC_i, \hat{d}_i)$ is trivial in positive degree.
\begin{lemma}[``Poincar\'e lemma'']
\[
H(\hcC_i, \hat{d}_i) = 
\hcC .
%\begin{cases} \mathcal{C}&k=0,\\ 0&k>0.\end{cases}
\]
\end{lemma}
\begin{proof}
A simple proof of this fact can be given in terms of an homotopy map, a procedure that we will use repeatedly in the following.
For fixed $i=1, \dots ,n$ and $s\geq1$, let
\begin{equation}
h_{i,s} = \frac{\partial }{\partial \theta_i^{s+1}} \int du^{i,s},
\end{equation}
where the integration constant is set to zero. We have
\begin{equation}
h_{i,s} \hat{d}_i + \hat{d}_i h_{i,s} = 1-\pi_{u^{i,s}} \pi_{\theta_i^{s+1}}
\end{equation}
where $\pi_p$ denotes the projection that sets the variable $p$ to zero. 
Given a cocycle $g \in\hcC_i$ representing a cohomology class $[g]$, the previous formula implies that the same cohomology class can be represented by the cycle $\pi_{u^{i,s}} \pi_{\theta_i^{s+1}} g$. Repeating this process, we can kill all variables $u^{i,s}$, $\theta_i^{s+1}$ with $s\geq1$, hence a cohomology class can be always represented by an element in $\hcC$. Since $\im \hat{d}_i$ always contains $\theta_i^{s+1}$ with $s\geq 1$, no further simplification is possible.
\end{proof}
\begin{proof}[Proof of Proposition~\ref{Hd0}]

To prove Proposition~\ref{Hd0} we introduce two homotopy maps. 
Let $\sigma_i$ be the map that acts on a rational function of $\lambda$ by removing the polar part at $\lambda=u^i$. For a polynomial $p(\lambda)$ we have
\begin{equation}
\sigma_i\left( \frac{p(\lambda)}{\lambda- u^i} \right) = \frac{p(\lambda) - p(u^i) }{\lambda-u^i} .
\end{equation}
Fix $i=1,\dots ,n$ and $s\geq1$ and let
\begin{equation}
h_{i,s}=\sigma_i \frac1{u^i - \lambda}\frac1{f^i} \frac{\partial }{\partial \theta_i^{s+1}} \int du^{i,s} .
\end{equation}
Clearly $h_{i,s}$ defines a map on $\hcA[\lambda]$ which satisfies
\begin{equation}
h_{i,s} \Delta_{-1} + \Delta_{-1} h_{i,s} =  1- p_{i,s},   \label{19}
\end{equation}
where
\begin{align}
p_{i,s} :=  &\pi_{u^{i,s}} \pi_{\theta_i^{s+1}} + \Big(
  \sum_{\substack{ t\geq1\\ j} } \frac{f^j}{f^i}  \theta_j^{t+1} \frac{\partial }{\partial \theta_i^{s+1}} \frac{\partial }{\partial u^{j,t}} \int d u^{i,s}   \Big) \pi_{\lambda -u^i} .
\end{align}
As before $\pi_p$ denotes the projection that sets the variable $p$ to zero, in particular $\pi_{\lambda-u^i}$ sets $\lambda$ to $u^i$. 
% above formulas checked again 2/8/16

In the derivation of formula~\eqref{19} is useful to notice the following obvious identities
\begin{equation}
  \int d u^{i,s} \frac{\partial }{\partial u^{i,s}} = 1 - \pi_{u^{i,s}} ,
\qquad 
\frac{\partial }{\partial \theta_i^{s+1}} \theta_i^{s+1} = \pi_{\theta_i^{s+1}} .
\end{equation}

%\note{In the derivation of the anticommutator above, it is useful to notice the following identities:
%\begin{equation}
%[\frac{\partial }{\partial u^{i,s}} ,\int d u^{j,t} ] = \delta_{i,j} \delta_{s,t} \pi_{u^{i,s}} ,
%\end{equation}
%\begin{equation}
%\frac{\partial }{\partial u^{i,s}} \int d u^{i,s} = 1, \qquad
% \int d u^{i,s} \frac{\partial }{\partial u^{i,s}} = 1 - \pi_{u^{i,s}} ,
%\end{equation}
%\begin{equation}
%\frac{\partial }{\partial \theta_i^{s+1}} \theta_i^{s+1} = \pi_{\theta_i^{s+1}} .
%\end{equation}
%}
%
%\note{
%Note that we have an alternative expression for $p_{i,s}$
%\begin{align}
%p_{i,s} :=  &\pi_{u^{i,s}} \pi_{\theta_i^{s+1}} + \Big(
%1- \pi_{u^{i,s}} \pi_{\theta_i^{s+1}} +  \notag \\
%& - \sum_{\substack{ t\geq1\\ j} } \frac{f^j}{f^i} \frac{\partial }{\partial \theta_i^{s+1}} \theta_j^{t+1} \int d u^{i,s} \frac{\partial }{\partial u^{j,t}}  \Big) \pi_{\lambda -u^i} .
%\end{align}
%}

%
%\note{A second alternative version for $p_{i,s}$:
%\begin{align}
%p_{i,s} :=  &\pi_{u^{i,s}} \pi_{\theta_i^{s+1}} + \Big(
%\pi_{u^{i,s}}- \pi_{u^{i,s}} \pi_{\theta_i^{s+1}} +  \notag \\
%& + \sum_{\substack{ t\geq1\\ j} } \frac{f^j}{f^i} \theta_j^{t+1}  \frac{\partial }{\partial \theta_i^{s+1}} \int d u^{i,s} \frac{\partial }{\partial u^{j,t}}  \Big) \pi_{\lambda -u^i} .
%\end{align}
%}

The operator $p_{i,s}$ is homogeneous of standard degree zero, hence it acts separately on each $\hcA_d[\lambda]$, where the infinite sum appearing in its definition becomes finite, hence it is well defined. Moreover for $s$ large enough, it acts like the identity. We therefore introduce a well-defined operator $p_I$ on $\hcA[\lambda]$, which is given by the application of all the operators $p_{i,s}$ in a given order, i.e.,
\begin{equation}
p_I = \cdots \circ (p_{n,2} \circ \cdots \circ p_{1,2}) \circ (p_{n,1} \circ \cdots \circ p_{1,1} ).
\end{equation}										
It is easy to check that the operator $p_I$ maps $\hcA[\lambda]$ to $\hcA^{nt} \oplus\hcC[\lambda]$, where by $\hcA^{nt}$ we denote the subspace of $\hcA$ spanned by monomials with non-trivial dependence on the variables $u^{i,s}$, $\theta_i^{s+1}$ with $i=1,\dots ,n$, $s\geq1$.

By formula~\eqref{19} when applied to a $\Delta_{-1}$-cocycle the operator $p_I$ produces an equivalent cocycle $\tilde{g} + g$ with $\tilde{g}\in \hcA^{nt}$ and $g \in  \hcC[\lambda]$. 
%It follows that the operator $p_I$ kills the dependence on all the variables $u^{i,s}$, $\theta_i^{s+1}$ with $i=1,\dots ,n$, $s\geq1$, in the $\lambda$-dependent part of any cocyle.
%In other words a cohomology class  can always be represented by a cocycle of the form $g + \tilde{g}$, where $g \in \hcC[\lambda]$ and $\tilde{g}\in\hat{\mathcal{A}}$, and we require each monomial in $\tilde{g}$ to have a nontrivial dependence on the variables $\theta_i^{s+1}$ for $s\geq1$.
Notice that $g$ is a cocycle by itself, and can no longer be simplified by quotienting by $\im \Delta_{-1}$.

Let us now introduce a second homotopy operator. Let us denote $\Delta_{-1} = d'' - \lambda d'$ and define, for $s,t\geq1$ and $i\not= j$
\begin{equation}
h_{i,s;j,t} = \frac1{u^i-u^j} \frac1{f^i f^j} \frac{\partial}{\partial \theta_i^{s+1}} \frac{\partial }{\partial \theta_j^{t+1}}  \int du^{i,s} \int du^{j,t} .
\end{equation}
We have
\begin{equation}
[ h_{i,s;j,t}, d'' d' ] = 
%(1- \pi_{u^{i,s}} \pi_{\theta_i^{s+1}})(1-\pi_{u^{j,t}} \pi_{\theta_j^{t+1}}) 
1- p_{i,s;j,t}
+ (...) d' + (...)d'',
\end{equation}
where we did not specify the last two terms since they vanish when applied on elements in $\ker d' \cap \ker d''$, and
\begin{equation}
p_{i,s;j,t} :=  \pi_{u^{i,s}} \pi_{\theta_i^{s+1}} + \pi_{u^{j,t}} \pi_{\theta_j^{t+1}}-  \pi_{u^{i,s}} \pi_{\theta_i^{s+1}}\pi_{u^{j,t}} \pi_{\theta_j^{t+1}}.
\end{equation}
Clearly $p_{i,s;j,t}$ is a well-defined operator on $\hcA[\lambda]$, which acts like the identity for large $s$ or $t$. As before we define an operator $p_{II}$ on $\hcA[\lambda]$ given by the application of the such operators in a given order. The operator $p_{II}$ acts like the identity on $\hcC[\lambda]$ and maps $\hcA^{nt}$ to $\oplus_{i} \hcC_i^{nt}$, where $\hcC_i^{nt} = \hcC_i \cap \hcA^{nt}$, since it obviously removes all monomials that contain any quadratic term of the form $u^{i,s} u^{j,t}$, $u^{i,s} \theta_j^{t+1}$ or $\theta_i^{s+1} \theta_j^{t+1}$ for $i\not= j$, $s,t\geq1$. We denote by $\hcM$ the subspace of $\hcA$ spanned by such monomials. Recall that $\Delta_{-1} =\sum_i (-\lambda+u^i) f^i \hat{d}_i$ and observe that $\hat{d}_i$ sends to zero $\hcC[\lambda]$ and $\hcC_j^{nt}[\lambda]$ for $j \not=i$. Moreover $\hat{d}_i (\hcM) \subseteq \hcM$. Notice that this in particular implies that each summand in 
\begin{equation*}
\hcA[\lambda] = \hcC[\lambda] \oplus  \left( \oplus_i \hcC_i^{nt}[\lambda] ) 
\right) \oplus \hcM[\lambda],
\end{equation*}
is left invariant by $\Delta_{-1}$. 

Let us then consider the cocycle $\tilde{g}\in\hcA^{nt}$. Since $\tilde{g}$ does not depend on $\lambda$, it is in the kernel of $\Delta_{-1}$ iff $d' \tilde{g} =0$ and $d'' \tilde{g} =0$.
%\begin{equation}
%d' = \sum f^i \theta_i^{s+1} \frac{\partial }{\partial u^{i,s}},
%\quad
%d'' = \sum u^i f^i \theta_i^{s+1} \frac{\partial }{\partial u^{i,s}} .
%\end{equation}
Moreover, an element of the form $d'' d' f$ for $f\in\hat{\mathcal{A}}$ is in the image of $\Delta_{-1}$ and does not depend on $\lambda$, since 
$\Delta_{-1} d'f = d'' d' f $.
By the homotopy formula above it follows that $p_{II}$ sends the cocycle $\tilde{g}$ to an equivalent cocycle $\sum_i \tilde{g}_i$ for $\tilde{g}_i \in \hcC_i^{nt}$.

Moreover, since $\hat{d}_i( \hcC_i^{nt} ) \subseteq \hcC_i^{nt}$, and $\Delta_{-1} =\sum_i (-\lambda+u^i) f^i \hat{d}_i$, we have that $\tilde{g}_i \in \ker \hat{d}_i$, which in turn, because of the vanishing of the cohomology $H(\hcC_i, \hat{d}_i )$, is equivalent to  $\tilde{g}_i \in \im \hat{d}_i \subseteq \hcC_i^{nt}$. 

%It follows that, up to elements in $\im d'' d'$, a cocycle $\tilde{g}$ is equivalent to the cocycle $p_{i,s;j,t} (\tilde{g})$.
%%\begin{equation}
%%(\pi_{u^{i,s}} \pi_{\theta_i^{s+1}} + \pi_{u^{j,t}} \pi_{\theta_j^{t+1}} - \pi_{u^{i,s}} \pi_{\theta_i^{s+1}}\pi_{u^{j,t}} \pi_{\theta_j^{t+1}}) \tilde{g} .
%%\end{equation}
%This amounts to removing all monomials that contain any quadratic term of the form $u^{i,s} u^{j,t}$, $u^{i,s} \theta_j^{t+1}$ or $\theta_i^{s+1} \theta_j^{t+1}$ for $i\not= j$, $s,t\geq1$.
%Hence $\tilde{g}$ can be written as sum of $g_i \in \hcC_i$ for $i=1,\dots ,n$, where  each monomial in $g_i$ has a nontrivial dependence on the variables $\theta_i^{s+1}$ for $s\geq1$. 
%

%Moreover, since $\Delta_{-1} =\sum_i (\lambda-u^i) f^i \hat{d}_i$ and $\hat{d}_i$ acts nontrivially only on $\hcC_i$, we have that $g_i \in \ker \hat{d}_i$. Because of the vanishing of the cohomology $H(\hcC_i, \hat{d}_i )$ we can simply require that $g_i \in \im \hat{d}_i$, since that will automatically take care of the nontrivial dependence on $\theta_i^{s+1}$.

To complete the proof we have to show that we cannot further quotient by elements in $\im \Delta_{-1}$ without spoiling the form of $\sum_i\tilde{g}_i$. It is sufficient to show that if we have a $\Delta_{-1}$-coboundary of the form $\sum_i g_i$ with $g_i \in \hcC_i^{nt}$, then the $g_i$ vanish. Let, then, $h(\lambda) \in \hcA[\lambda]$ such that $\Delta_{-1} h(\lambda) = \sum_i g_i$ with $g_i \in \hcC_i^{nt}$. Since $\Delta_{-1}$ leaves invariant each summand in the splitting of $\hcA[\lambda]$ above, we can choose $h(\lambda) = \sum_i h_i(\lambda)$ for $h_i(\lambda) \in \hcC_i^{nt}[\lambda]$, where the $h_i(\lambda)$ have to satisfy $f_i (-\lambda + u^i) \hat{d}_i h_i(\lambda) = g_i$. Since $g_i$ does not depend on $\lambda$ this clearly implies that it has to be zero. We can conclude that the we cannot further simplify the cocycle $\sum_i \tilde{g}_i$.

%In principle the cocycle $g_i \in \im \hat{d}_i$ could be quotiented by an element of the form $\Delta_{-1} f(\lambda) \in \hcC_i$ for $f(\lambda) \in \hcA[\lambda]$. It is easy to see that $f \in \hcC_i[\lambda]$, otherwise the image $\Delta_{-1} f$ would also depend on some variables $u^{j,s}$, $\theta_j^{s+1}$ for $j\not= i$. We have then
%\begin{equation}
%\Delta_{-1} f = (u^i-\lambda) f^i \sum_{s\geq1} \theta_i^{s+1} \frac{\partial f}{\partial u^{i,s}} . 
%\end{equation}
%But such element can be independent of $\lambda$ iff it vanishes. 

Proposition~\ref{Hd0} is proved.

\end{proof}

We now continue to determine the differential on the first page of the spectral sequence. Define the space 
\begin{equation}
\label{decB}
\hcB:=\hcC[\lambda] \oplus \bigoplus_{i=1}^n \frac{\im  \left( \hat{d}_i : \hcC_i \to \hcC_i \right)[\lambda]}{(\lambda-u^{i})},
\end{equation}
where, in general, by $\frac{A[\lambda]}{(\lambda-u^i)}$ we denote the quotient of the space of polynomials $A[\lambda]$ by the space of polynomials that vanish at $\lambda = u^i$.

The inclusion $i:\hcB\hookrightarrow\hcA[\lambda]$ is defined as the identity on the first component $\hcC[\lambda]$, and, on $\hat{d}_i(\hcC_{i})[\lambda]/(\lambda-u^i)$, as the standard inclusion after the evaluation at $\lambda=u^{i}$. 
Proposition~\ref{Hd0} says that this inclusion $i:(\hcB,0)\hookrightarrow (\hcA[\lambda],\Delta_{-1})$ induces an isomorphism on cohomology.  The differential on the first page of the spectral sequence is just the 
differential induced on this cohomology by the next degree term $\Delta_{0}$ of $D_{\lambda}$.  More explicitly, we use $i$
 to embed elements $b\in\hcB$ in $\hcA[\lambda]$ as cocycles for $\Delta_{-1}$, and apply $\Delta_{0}$. Because $D_{\lambda}^{2}=0$, we have 
$\Delta_{-1}\Delta_{0}+\Delta_{0}\Delta_{-1}=0$, and the result  $\Delta_{0}i(b)$ is another cocycle for $\Delta_{-1}$, however not in the image of $i$. Therefore, to project down to the image of $i$, we use the 
homotopies described in the proof of Proposition~\ref{Hd0}. Let us denote by $p := p_{II} \circ p_I$ the map on $\hcA[\lambda]$ defined by the composition $p_I$ of all the maps $p_{i,s}$ followed by the composition $p_{II}$ of all the maps $p_{i,s;j,t}$, defined in the proof of Proposition~\ref{Hd0}. 
In particular the map $p$ associates with a $\Delta_{-1}$ cocycle in $\hcA[\lambda]$ the representative in $\hcB$ of the corresponding cohomology class. We obviously have $p\circ i={\rm id}$.  Then we have:
\begin{proposition}
The first page of the spectral sequence associated with the filtration $F \hcA[\lambda]$ is given by the complex
\[
(E_{1},d_{1})=(\hcB,\Delta_{0}'),
\]
with $\Delta_{0}'=p\circ\Delta_{0}\circ i$.
\end{proposition}

From the proof of Proposition~\ref{Hd0} we can derive some properties of the map $p$, which we collect in the following Corollary. 
%{It is important to define $p$ on the whole space $\hcA[\lambda]$ rather than just on cocyles, see the proof of Lemma~\ref{diff}.}
Let us write the space $\hcA[\lambda]$ as the direct sum of three  subspaces, i.e.,
\def\hcM{\hat{\mathcal{M}}}
\begin{equation}
\hcA[\lambda] = \hcC[\lambda] \oplus \left( \oplus_i \hcC_i^{nt}[\lambda] \right) \oplus \hcM[\lambda],
\end{equation}
where, as already mentioned in the proof of Proposition~\ref{Hd0}, $\hcC_i^{nt}$ is the subspace of $\hcC_i$ generated by all nontrivial monomials, i.e., monomials that contain at least one generator among $u^{i,s}$, $\theta_i^{s+1}$ for $s\geq1$, and $\hcM$ is the subspace of $\hcA$ generated by all nontrivial mixed monomials, i.e., monomials containing at least a pair of generators with different indices $i$. 
We further decompose
\begin{equation}
\hcC_i^{nt}[\lambda] = \hcC_i^{nt} \oplus	(\lambda-u^i)\hcC_i^{nt}[\lambda] .
\end{equation}

\begin{corollary} \label{pcor}
The map $p$ acts as the identity on $\hcC[\lambda]$ and $\hcC_i^{nt} \cap \ker \hat{d}_i$, and it sends to zero the spaces $(\lambda - u^i) \hcC^{nt}_i[\lambda]$, $i=1,\dots ,n$, and $\hcM[\lambda]$. 
\end{corollary}
\begin{proof}

The projection $p$ on $\hcA[\lambda]$ is defined as $p_{II} \circ p_{I}$, where $p_I$ is given by the composition of the operators $p_{i,s}$ for $s\geq1$ and $p_{II}$ by the composition of the operators $p_{i,s;j,t}$ for $i\not=j$, $s,t\geq1$.

It is quite easy to check that $p_I = 1$ on $\hcC[\lambda]$, since all $p_{i,s}$ act just like the identity. 

On $\hcC_i^{nt}[\lambda]$ the operator $p_{k,s}$ acts like the identity  unless $i=k$. On $(\lambda - u^i) \hcC^{nt}_i[\lambda]$ the operator $\pi_{\lambda-u^i}$ is equal to zero, so $p_{i,s} = \pi_{u^{i,s}} \pi_{\theta_i^{s+1}}$, therefore $p_I =0$ on $(\lambda-u^i) \hcC_i^{nt} [\lambda]$.

Notice that on $\hcC_i^{nt}$ the operator $p_{i,s}$ becomes
\begin{equation}
p_{i,s} = 1- \frac{\partial }{\partial \theta_i^{s+1}} \int du^{i,s} \hat{d}_i
\end{equation}
while $p_{j,s} = 1$ for $j\not=i$, therefore on $\hcC_i^{nt} \cap \ker \hat{d}_i$, the operator $p_I$ acts like the identity. 
Moreover notice that $p_I(\hcM[\lambda]) \subset \hcM[\lambda]$.

Finally it is easy to see that $p_{II}$ acts as the identity on $\hcC[\lambda]$ and $\hcC_i^{nt}[\lambda]$, and as the zero operator on $\hcM[\lambda]$. The Corollary is proved.
\end{proof}

\subsection{The second filtration}
We are now left with the problem of computing the second page $E_2$ of the spectral sequence associated with the filtration $F \hcA[\lambda]$, which amounts, as explained in the previous section, to computing the cohomology of the operator $\Delta_0'$ on the space $\hcB$. This task is not trivial, so we introduce a filtration on $\hcB$ and consider the associated spectral sequence. 

Let us define a degree $\deg_{\theta^1}$ on $\hcA$, and consequently on $\hcB$, that simply counts the number of $\theta_i^1$, $i=1,\dots ,n$. The differential $\Delta_0$ splits in three homogeneous components 
\begin{equation}
\Delta_0 = \Delta_{0,1} + \Delta_{0,0} + \Delta_{0,-1}
\end{equation}
where $\deg_{\theta^1} \Delta_{0,k} = k$. It is clear from their definition that the maps $i$ and $p$ preserve the degree $\deg_{\theta^1}$, hence the homogeneous components of $\Delta'_0$ are simply given by $\Delta'_{0,k} = p \circ \Delta_{0,k} \circ i$.

%\note{It is easy to see that the term $\Delta_{0,2}$ is zero by looking carefully at the formula for $\Delta_0$: a degree two term can only appear if we set $b=0$ and $a=1$ in the last three lines of such expression, but in such case we get $s=a+b=1$, hence a derivative in $\theta_i^{1}$ appears, decreasing the degree.
%}

Let us now introduce on $\hcB$ the decreasing filtration $F \hcB$ associated with the degree $\deg_{\theta^1} - \deg_\theta$, i.e., let $F^r \hcB$ be given by elements of $\hcB$ with homogeneous components in $\deg_{\theta^1} - \deg_\theta$ of degree less or equal $-r$,
\[
\ldots\subset F^2\hcB\subset F^1\hcB\subset F^0\hcB=\hcB.
\]

This filtration is preserved by $\Delta'_0$, hence this turns $(\hcB,\Delta_0')$ into a filtered complex to which we associate a spectral sequence $(E'_k, d'_k)$. The zeroth page is given by $\hcB$ with induced differential given by $\Delta'_{0,1}$, i.e.,
\begin{equation}
(E'_0, d'_0) = ( \hcB, \Delta'_{0,1}) . 
\end{equation}
Our aim is to compute the first page $E'_1$, i.e., the cohomology of $\Delta'_{0,1}$ on $\hcB$. Let us first compute the explicit expression of the differential.

\begin{lemma}
\label{diff}
The differential $\Delta'_{0,1}$ on $\hcB$ is given by 
\begin{equation}
\Delta'_{0,1} = p \circ \hat{\Delta}_{0,1} \circ i
\end{equation}
where
% ok checked, the correction in $\Delta_0,1$ in the proof does not affect this one 9/15
\begin{align}
\hat{\Delta}_{0,1} & = 
\quad(-\lambda+u^i)f^i\theta_i^{1}\frac{\d}{\d u^i} 
\\ \notag &
+\sum_{\substack{
s \geq 1
}}  \frac{s+2}{2}
f^i u^{i,s}\theta_i^1 
\frac{\d}{\d u^{i,s}}
\\ \notag & 
-\frac{1}{2}\sum_{\substack{
s \geq 1
}}
 (-\lambda+u^j) s 
f^j\frac{\d_j f^i}{f^i} u^{i,s}\theta_j^1
\frac{\d}{\d u^{i,s}}
\\ \notag
& - \frac{1}{2} 
(-\lambda+u^j)
\d_if^j  \theta_j^1 \theta_j^{0}
 \frac{\d}{\d \theta_i^0}
+ \frac{1}{2} \sum_{\substack{
s\geq 0
}}  
f^i  (s-1)\theta_i^1\theta_i^{s}
 \frac{\d}{\d \theta_i^s}
\\ \notag
&
-
\frac{1}{2} \sum_{\substack{
s\geq 0
}}  
(-\lambda+u^j)
f^j\frac{\d_j f^i}{f^i} ( s+1)\theta_j^1 \theta_i^s
 \frac{\d}{\d \theta_i^s}
\\ \notag &
+
\frac{1}{2} 
(-\lambda+u^j)
f^j\frac{\d_j f^i}{f^i}  \theta_i^{1} \theta_j^0
 \frac{\d}{\d \theta_i^0} .
\end{align}
\end{lemma}
\begin{proof}
Let us first collect the terms in $\Delta_0$ that increase the degree $\deg_{\theta^1}$. We get:
% there was an extra s=1 term in the last two lines of the following, that I have corrected, 28/8
\begin{align}
\Delta_{0,1} & = 
(-\lambda+u^i)f^i\theta_i^{1}\frac{\d}{\d u^i} 
\\ \notag &%%%%%%% 2
+
\sum_{\substack{
s \geq 1
}} (-\lambda+u^i) \partial_j f^i u^{j,s} \theta_i^{1} \frac{\d}{\d u^{i,s}}
+ \sum_{\substack{
s\geq 1
}} f^i u^{i,s} \theta_i^{1} \frac{\d}{\d u^{i,s}}
\\ \notag %%%%%%%%% 3
& +\frac{1}{2}\sum_{\substack{
s \geq 1
}}  (-\lambda+u^i) s
\d_jf^i u^{j,s} \theta_i^1
\frac{\d}{\d u^{i,s}}
+\frac{1}{2}\sum_{\substack{
s \geq 1
}}  s
f^i u^{i,s} \theta_i^1
\frac{\d}{\d u^{i,s}}
\\ \notag %%%%%%%%%%%% 4
& +\frac{1}{2}\sum_{\substack{
s \geq 1
}}  (-\lambda+u^i) s 
f^i\frac{\d_i f^j}{f^j} u^{j,s}\theta_j^1 
\frac{\d}{\d u^{i,s}}
\\ \notag %%%%%%%%%%5
& -\frac{1}{2}\sum_{\substack{
s \geq 1
}}
 (-\lambda+u^j) s 
f^j\frac{\d_j f^i}{f^i} u^{i,s}\theta_j^1
\frac{\d}{\d u^{i,s}}
\\ \notag
& + \frac{1}{2} \sum_{\substack{
s\geq 0
}}  
(-\lambda+u^j)
\d_if^j  (\theta_j^s \theta_j^{1} + s\theta_j^1 \theta_j^{s})
 \frac{\d}{\d \theta_i^s}
+ \frac{1}{2} \sum_{\substack{
s\geq 0
}}  
f^i  (\theta_i^s\theta_i^{1} +s\theta_i^1\theta_i^{s})
 \frac{\d}{\d \theta_i^s}
\\ \notag
%&
%+ \frac{1}{2} \sum_{\substack{
%s\geq 0
%}}  
%(-\lambda+u^j)
%f^j\frac{\d_j f^i}{f^i} (\theta_i^s \theta_j^{1}+s\theta_i^1 \theta_j^{s})   \frac{\d}{\d \theta_i^s}
%\\ \notag
%&
%- 
%\frac{1}{2} \sum_{\substack{
%s\geq 0
%}}  
%(-\lambda+u^j)
%f^j\frac{\d_j f^i}{f^i} (\theta_j^s \theta_i^{1} + s\theta_j^1 \theta_i^s)
% \frac{\d}{\d \theta_i^s} \\
%corrected version:
& +\frac12  
\sum_{\substack{
s\geq 0 \\ s\not=1
}}  
(-\lambda+u^j)
f^j\frac{\d_j f^i}{f^i} 
(s+1) (\theta_i^s \theta_j^1 - \theta_j^s \theta_i^1 ) 
\frac{\d}{\d \theta_i^s} \\
& +(-\lambda+u^j)
f^j\frac{\d_j f^i}{f^i}  \theta_i^1 \theta_j^1 
\frac{\d}{\d \theta_i^1} .
\end{align}
Since $\Delta_{-1}$ has $\deg_{\theta^1}$ equal to zero, the equation 
$\Delta_0 \Delta_{-1} + \Delta_{-1} \Delta_0 =0$ implies in particular that
\begin{equation}
\Delta_{0,1} \Delta_{-1} + \Delta_{-1} \Delta_{0,1} =0 .
\end{equation}
Since the map $i$ maps to the kernel of $\Delta_{-1}$, it follows from the previous equation that $\Delta_{0,1} \circ i$ also maps to the kernel of $\Delta_{-1}$.

%The second 
%equation shows that since $i$ maps to the kernel of $\Delta_{-1}$, the projection $p$ acting on 
%$\hcC_i[\lambda]$ is always given by setting $\lambda=u_i$ in the computation of 
%$\Delta_0':=p\Delta_0 i$.

%\note{The point here is that while $\Delta_{0,1} \circ i$ is in the kernel off $\Delta_{-1}$, the same cannot be said for each line in the expression above. That is why Corollary 14 has to be given for a general element in $\hcA[\lambda]$ rather than just for cocycles.}

Keeping in mind the properties of the map $p$ outlined in Corollary~\ref{pcor}, let us now determine which of the terms in the formula for 
$\Delta_{0,1}$ above, applied to an element in the image of $i$, are mapped to zero when composing with $p$. 

The second, fourth and sixth terms are of the form
\begin{equation}
\sum_{s\geq1} C_s^{ij} (\lambda - u^i) u^{j,s} \frac{\partial }{\partial u^{i,s}} 
\end{equation}
for $C_s^{ij} \in \hcC$. When evaluated on the image 
of \eqref{decB} through $i$, such operator clearly gives $0$ on $\hcC[\lambda]$. An element
%\note{\footnote{\note{The operator in the displayed formula when acting on $\hat{d}_k(\hcC_k)$ is clearly non-zero only if $i=k$. Moreover, recall that any monomial in $\hat{d}_k(\hcC_k)$ is non-trivial in $\theta_i^{s+1}$, from some $s\geq1$. }}}
 of $\hat{d}_i (\hcC_i)$ is mapped, for $i\not=j$, to $\hcM[\lambda]$, which in turn is mapped to zero by $p$. 
Finally, for $i=j$, such element is mapped again to an element in $\hcC_i$, which however vanishes for $\lambda=u^i$, hence is mapped to zero by $p$. Hence these terms do not contribute to $\Delta'_{0,1}$. 

The eighth term vanishes identically for $s=1$ and for $s\geq 2$ either maps $\hcC_i$ to $\hcM[\lambda]$, in which case applying $p$ gives zero, or maps $\hcC_i$ 
to $\hcC_j$ so that setting $\lambda=u_j$ yields zero. Therefore only the $s=0$ term contributes.

In the tenth term, observe that the operator $(-\lambda+u^j)\theta_j^s$, $s\geq2$ acting on 
$\hcC_j$ gives zero when composing with $p$. Collecting all the remaining terms, we obtain the expression stated in the Lemma.
\end{proof}
% checked again 3/8/16

\subsection{The third filtration}

We have arrived at the problem of computing the cohomology of the complex  $(\hcB,\Delta'_{0,1})$.
In this Section first we show that we can compute such cohomology on each of the $n+1$ summands in~\eqref{decB} independently, and then 
we introduce our final filtration to estimate the possible non-vanishing $(p,d)$-degrees in the cohomology of $\Delta_{0,1}'$.

Let us start with the following observation:
\begin{lemma}
The operator $\Delta'_{0,1}$ preserves the splitting \eqref{decB} of the 
cohomology $\hcB$ of $\Delta_{-1}$ in $n+1$ summands, namely
\begin{align*}
 \Delta'_{0,1} \left(\hat{\mathcal{C}}[\lambda]\right)& \subset \hat{\mathcal{C}}[\lambda], \\ 
 \Delta_{0,1}' \left(\frac{\im  \left( \hat{d}_i : \hcC_i \to \hcC_i \right)[\lambda]}{(u^i-\lambda)}\right) &
 \subset \frac{\im  \left( \hat{d}_i : \hcC_i \to \hcC_i \right)[\lambda]}{(u^i-\lambda)},
\end{align*}
for $i=1,\dots ,n$.
\end{lemma}
\begin{proof}
Acting on $\hcA[\lambda]$, one explicitly checks that each term of the operator $\hat\Delta_{0,1}$  as stated in Lemma \ref{diff} preserves the subspaces  $\hcC[\lambda]$ as well as $\hcC_{i}^{nt}[\lambda]$. Since $p$ is the identity on $\hcC[\lambda]$, it is preserved by $\Delta_{0,1}'$, too.

As noticed before, we have that $\Delta_{0,1} \Delta_{-1} + \Delta_{-1} \Delta_{0,1} =0$. Since $\Delta_{-1} = f^i (-\lambda+u^i) \hat{d}_i$ on $\hcC_i^{nt}$, we get that $\Delta_{0,1}$ sends $\hcC_i^{nt} \cap \ker \hat{d}_i = \hat{d}_i ( \hcC_i )$ to itself. Finally $p$ acts like the identity on such space, which is therefore preserved by $\Delta_{0,1}'$.
%Taking the $\theta^{1}$ and $u^{i,s\geq 1}$-component of degree $1$ resp. $-1$ of the equation $\Delta_{0,1} \Delta_{-1} + \Delta_{-1} \Delta_{0,1} =0$ derived above, we  see
%that $\Delta_{0,1}$ anticommutes with $\hat{d}_{i}$, and therefore maps ${\rm Im}(\hat{d}_{i})$ to itself. With this, the statement of the Lemma follows.
\end{proof}

This Lemma implies that we can compute the cohomology of $\Delta'_{0,1}$ on each of these $n+1$ spaces separately. Let us start from $\hcC[\lambda]$. 
\begin{lemma}
The cohomology of the complex $(\hcC[\lambda], \Delta'_{0,1})$ vanishes if the bi-degrees $(p,d)$ are not in the range
\begin{equation}
d=0,\dots ,n, \qquad p=d,\dots ,d+n .
\end{equation}
\end{lemma}
\begin{proof}
The possible $(p,d)$-degrees of the elements of $\hat{\mathcal{C}}[\lambda]$ are those of monomials $\theta^0_{i_1}\cdots \theta^0_{i_k}\theta^1_{j_1}\cdots \theta^1_{j_\ell}$, $1\leq i_1<\cdots<i_k\leq n$, $1\leq j_1<\cdots<j_\ell\leq n$. So, for the standard gradation $d$ we have $0\leq d\leq n$, and, if we fix $d$, then for the super gradation $p$ we have $d\leq p\leq d+n$. 
%Without any further computation we see that these are the restrictions for the possible $(p,d)$-degrees of the cohomology of $\Delta'_{01}$.
\end{proof}

 Note that this Lemma gives  Case 1 in the statement of Remark~\ref{rmk:cases}. 

Now we want to estimate the cohomology of $\Delta'_{0,1}$ in each of the spaces $\im  \left( \hat{d}_i : \hcC_i \to \hcC_i \right)[\lambda]/(u^i-\lambda)$, $i=1,\dots,n$. For the rest of this Section the index $i$ is fixed and refers to the particular space that we consider. 

Obviously, the map $\pi_{\lambda-u^i}$ that sets $\lambda$ to $u^i$ defines an isomorphism between $\im  \left( \hat{d}_i : \hcC_i \to \hcC_i \right)[\lambda]/(u^i-\lambda)$ and $\im  \left( \hat{d}_i : \hcC_i \to \hcC_i \right)$. Let us denote by $\tilde{\Delta}_{0,1}'$ the operator induced by $\Delta_{0,1}'$ on $\hat{d}_i(\hcC_i)$ by such isomorphism and let us represent it in the following way:
$$
\tilde{\Delta}'_{0,1}  = \sum_{\substack{k=1}}^n \theta_{k}^1 \mathcal{D}_k,
$$
where
% checked again from \Delta'_{0,1} 8/9 ok
\begin{align}
\mathcal{D}_k:=&  
(u^k-u^i)f^k\frac{\d}{\d u^k} 
- \sum_{j=1}^n
(u^k-u^i)
f^k\frac{\d_k f^j}{f^j} \theta_j^1
\frac{\d}{\d \theta_j^1}
\\ \notag &
+\sum_{\substack{
		s \geq 1
	}}  \frac{s+2}{2}
	f^k u^{k,s} 
	\frac{\d}{\d u^{k,s}}
	-\frac{1}{2}\sum_{\substack{
			s \geq 1
		}}
		\sum_{j=1}^n
		(u^k-u^i) s 
		f^k\frac{\d_k f^j}{f^j} u^{j,s}
		\frac{\d}{\d u^{j,s}}
		\\ \notag &
		+ \frac{1}{2} \sum_{\substack{
				s\geq 2
			}}  
			f^k  (s-1) \theta_k^{s}
			\frac{\d}{\d \theta_k^s}
	%		\\ \notag &
			- \frac{1}{2} \sum_{\substack{
					s\geq 2
				}}  
				\sum_{j=1}^n
				(u^k-u^i)
				f^k\frac{\d_k f^j}{f^j} ( s+1) \theta_j^s
				\frac{\d}{\d \theta_j^s}
				\\ \notag
				& - \frac{1}{2} 
				\sum_{j=1}^n
				(u^k-u^i)
				\d_jf^k   \theta_k^{0}
				\frac{\d}{\d \theta_j^0}
		%		\\ \notag &
				-\frac{1}{2}  
				f^k  \theta_k^{0}
				\frac{\d}{\d \theta_k^0}
				\\ \notag
				&
				-
				\frac{1}{2} \sum_{j=1}^n  
				(u^k-u^i)
				f^k\frac{\d_k f^j}{f^j} \theta_j^0
				\frac{\d}{\d \theta_j^0}
		%		\\ \notag &
				+
				\frac{1}{2} \sum_{j=1}^n
				(u^j-u^i)
				f^j\frac{\d_j f^k}{f^k}  \theta_j^0
				\frac{\d}{\d \theta_k^0}.
				\end{align}
In particular
% checked again from previous formula 8/9 ok
\begin{align}
\label{diffi}
\mathcal{D}_i& =  
				\sum_{\substack{
						s \geq 1
					}}  \frac{s+2}{2}
					f^i u^{i,s} 
					\frac{\d}{\d u^{i,s}}
					+ \sum_{\substack{
							s\geq 2
						}}  \frac{s-1}2
						f^i   \theta_i^{s}
						\frac{\d}{\d \theta_i^s}
						\\ \notag
						&-\frac{1}{2}  
						f^i  \theta_i^{0}
						\frac{\d}{\d \theta_i^0}
						+
						\frac{1}{2} \sum_{j=1}^n
						(u^j-u^i)
						f^j\frac{\d_j f^i}{f^i}  \theta_j^0
						\frac{\d}{\d \theta_i^0} .
\end{align}
As a first step towards the computation of the cohomology of $\tilde{\Delta}'_{0,1}$ on the space $\hcB_i:=\hat{d}_i(\hcC_i)$, we introduce the decreasing filtration $F \hcB_i$ associated with the degree $\deg_{\theta_i^1}-\deg_\theta$, in the same way as we did for the filtration $F \hcB$ before. In this case $\deg_{\theta^1_i}$ is the degree that counts the number of $\theta^1_i$ for {\it fixed} $i$. Denote by $(E''_k ,d''_k)$ the associated spectral sequence. Since $\deg_{\theta_i^1} \mathcal{D}_k =0$, the zeroth page is given by $\hcB_i$ with the differential induced by $\theta^1_i \mathcal{D}_i$, i.e.,
\begin{equation}
(E''_0,d''_0) =(\hcB_i, \theta_i^1 \mathcal{D}_i ) .
\end{equation}
% ok checked everything up to the end 17/9
We now proceed to find the first page of the spectral sequence by computing the cohomology of this complex. We will find restrictions on the possible $(p,d)$-degrees at which the cohomology can be non-trivial, which will be sufficient to complete the proof of our vanishing Theorem~\ref{thm:A}.
\begin{proposition}
The cohomology of the differential $\theta_i^1 \mathcal{D}_i$ on the space 
$\hcB_i=\hat{d}_i (\hcC_i)$ vanishes,
\[
H^p_d\left( \hcB_i,\theta_i^1 \mathcal{D}_i\right)=0,
\]
unless $d=2,3,\ldots,n+2,~p=d,d+1,\ldots,d+n-1$.
\end{proposition}
\begin{proof}					
In order to compute the cohomology of $\theta_i^1 \mathcal{D}_i$ on $\im  \left( \hat{d}_i : \hcC_i \to \hcC_i \right)$, we represent this space as a direct sum of subcomplexes indexed by some auxiliary gradation. That is, we consider all possible non-constant polynomials in $u^{i,\geq1}$ and $\theta_i^{\geq 2}$, and we represent this space as a direct sum 
$$
\mathbb{C}[u^{i,\geq1},\theta_i^{\geq 2}]=\bigoplus_{w\in \frac 12\Z} M_i^w,
$$ 
where $M_i^w\subset \mathbb{C}[u^{i,\geq1},\theta_i^{\geq 2}]$ is the subspace of weighted homogeneous polynomials of the weight $w$, where the weight $w$ is defined on generators as $w(u^{i,s})=(s+2)/2$, $s=1,2,\dots$, and $w(\theta_i^s)=(s-1)/2$, $s=2,3,\dots$. 

\begin{lemma}
There is an isomorphism
\[
%\im  \left( \hat{d}_i : \hcC_i \to \hcC_i \right) 
\hcB_i \cong \bigoplus_{w\in \frac 12 \Z}  \hat{\mathcal{C}} \cdot \hat{d}_i(M_i^w) ,
\]
where each of the summands is preserved by the differential $\theta_i^1\mathcal{D}_i$. Therefore the cochain complex $(\hcB_i,\theta^1_i\mathcal{D}_i)$ decomposes into a direct sum of subcomplexes indexed by all possible $w\in\frac12\Z$.
%the variables $u^{i,\geq1}$ and $\theta_i^{\geq 2}$.
\end{lemma}
\begin{proof}
A short computation shows that
\[
[\mathcal{D}_i,\hat{d}_i]=- f^i\hat{d}_i.
\]
Consider an element $\mathfrak{m}\in M_i^w$. 
It follows from the identity above that
\begin{align*}
\theta^1_i\mathcal{D}_i(\hcC\hat{d}_i(\mathfrak{m}))
& \subset(\theta^1_i\mathcal{D}_i(\hcC))\cdot\hat{d}_i(\mathfrak{m})+\hcC\cdot\theta^1_i\mathcal{D}_i(\hat{d}_i(\mathfrak{m}))\\
&=(\theta^1_i\mathcal{D}_i(\hcC))\cdot\hat{d}_i(\mathfrak{m})+\hcC\cdot \theta^1_i \hat{d}_i(\mathcal{D}_i-f^i)(\mathfrak{m}).
\end{align*}
We see from equation \eqref{diffi} that $\theta^1_i\mathcal{D}_i(\hcC)\subset\hcC$
and that $(\mathcal{D}_i-f^i)$ acts on $\mathfrak{m}$ by multiplication by a scalar and by $f^i$. 
This shows that $\theta^1_i\mathcal{D}_i$ preserves the subspace $ \hat{\mathcal{C}} \cdot \hat{d}_i(\mathfrak{m})$, and, therefore, the space 
$ \hat{\mathcal{C}} \cdot \hat{d}_i(M_i^w)$ for any $w\in \frac12 \Z$.
\end{proof}
 Because of this Lemma, we can consider an infinite direct sum of complexes, each of which is a finite dimensional module over $\hcC$. Note that $\hat{d}_i(M_i^w)$ is equal to $0$ for $w<3/2$, so we assume that $w\geq 3/2$ in the rest of the proof.
 %isomorphic, as a bigraded vector space, to $\hcC$ with shifted $(p,d)$-gradation. The gradation is shifted by $(p_\mathfrak{m}+1,d_{\mathfrak{m}}+1)$, where $(p_\mathfrak{m},d_{\mathfrak{m}})$ in the gradation of $\mathfrak{m}\in M$ (and, therefore, $(p_\mathfrak{m}+1,d_{\mathfrak{m}}+1)$ is the gradation of $\hat{d}_i(\mathfrak{m})$).
 
Let us discuss the action of $\theta_i^1\mathcal{D}_i$. Observe that this operator is linear over the ring of functions in $u^1,\dots,u^n$ and $\theta_j^1$, $j\not=i$. So, we omit the coefficients from this ring in the computations below, assuming that there can be an arbitrary coefficient that would be preserved.

Note that the eigenvalue of the operator
$$
\sum_{\substack{
		s \geq 1
	}}  \frac{s+2}{2}
	 u^{i,s} 
	\frac{\d}{\d u^{i,s}}
	+ \sum_{\substack{
			s\geq 2
		}}  \frac{s-1}2
		  \theta_i^{s}
		\frac{\d}{\d \theta_i^s}
$$
on $\hat{d}_i(M_i^w)$ is equal to $w-1$. Then the eigenvalue of $\mathcal{D}_i$ on $\hat{d}_i(M_i^w)$ is $f^i (w-1)$. Observe that $w-1$ is always a positive half-integer, and the minimal value of $w-1$ is equal to $1/2$ for $w=3/2$ and in this case $\hat{d}_i(M_i^{3/2})=\langle \theta_i^{2}\rangle$. 

Consider monomials $\theta_{j_1}^0\cdots\theta_{j_\ell}^0$ such that $1\leq j_1<\ldots<j_\ell\leq n$, and $j_k\not=i$ for all $k=1,\dots,\ell$. Let $\mathfrak{m}\in M_i^w$. We have that
\begin{equation*}
\theta_i^1\mathcal{D}_i\colon \theta_{j_1}^0\cdots\theta_{j_\ell}^0 \hat{d}_i(\mathfrak{m})
\mapsto f^i(w-1)\theta_i^1 \theta_{j_1}^0\cdots\theta_{j_\ell}^0 \hat{d}_i(\mathfrak{m}) .
\end{equation*}
So, since $w-1\not=0$, we see that the subspace of $\hat{\mathcal{C}} \cdot \hat{d}_i(M_i^w)$ spanned by the elements $\theta_{j_1}^0\cdots\theta_{j_\ell}^0 \hat{d}_i(\mathfrak{m})$ and $\theta_i^1 \theta_{j_1}^0\cdots\theta_{j_\ell}^0 \hat{d}_i(\mathfrak{m})$ forms an acyclic subcomplex of $\hat{\mathcal{C}} \cdot \hat{d}_i(M_i^w)$.

Now we consider monomials $\theta_{i}^0\theta_{j_1}^0\cdots\theta_{j_\ell}^0$ such that $1\leq j_1<\cdots<j_\ell\leq n$, and $j_k\not=i$ for all $k=1,\dots,\ell$. Let $\mathfrak{m}\in M_i^w$. Modulo the acyclic subcomplex that we introduced in the previous paragraph, we have that
\begin{equation*}
\theta_i^1\mathcal{D}_i\colon \theta_i^0\theta_{j_1}^0\cdots\theta_{j_\ell}^0 \hat{d}_i(\mathfrak{m})
\mapsto f^i\left(w-\frac{3}{2}\right) \theta_i^1\theta_i^0 \theta_{j_1}^0\cdots\theta_{j_\ell}^0 \hat{d}_i(\mathfrak{m}).
\end{equation*}
Note that $w-3/2$ is equal to zero only if  $w=3/2$, and, therefore, $\hat{d}_i(\mathfrak{m})\not=\theta_i^2$. Thus, if  $w>3/2$, the quotient of $\hat{\mathcal{C}} \cdot \hat{d}_i(M_i^w)$ modulo an acyclic subcomplex is an acyclic subcomplex, and so $\hat{\mathcal{C}} \cdot \hat{d}_i(M_i^w)$ is acyclic. 

The only possible case when we can have non-trivial cohomology is the case of $w=3/2$, that is, the case of the complex $\hat{\mathcal{C}} \cdot \theta_i^2$. In this case, after taking the quotient modulo the acyclic subcomplex, the cohomology of $\theta_i^1\mathcal{D}_i$ is represented by a product of $\theta_i^0\theta_i^2$ by an arbitrary function in $u^1,\dots,u^n$, $\theta^0_1,\dots,\theta^0_n$ ($\theta_i^0$ is omitted), and $\theta^1_1,\dots,\theta^1_n$. This means that we have non-trivial cohomology only for gradation $d=2,3,\dots,2+n$, and once we fixed $d$, the possible values of gradation $p$ are $d,d+1,\dots,d+n-1$. This proves the Proposition.
\end{proof}
\begin{proof}[Proof of Theorem~\ref{thm:A}]
Recall from Remark \ref{rmk:cases} that we have to prove the vanishing of the bihamiltonian 
cohomology, except for the two cases 1 and 2 specified in that remark.
The computations in this subsection show that the cohomology of $\Delta_{01}'$ vanishes, unless
$d=0,1,\dots,n$; $p=d,d+1,\dots,d+n$ (Case 1, coming from the sub-complex $\hcC[\lambda]$), or 
$d=2,3,\dots,n+2$, $p=d,d+1,\dots,d+n-1$ (Case 2, coming from the sub-complex $\hat{d}_i(\hcC)$). 
Going back via the second to the first spectral sequence, we conclude the vanishing of the cohomology 
groups of the complex $(\hcA[\lambda],\Dla)$ in the same $(p,d)$-degrees: this is
exactly the statement of Theorem~\ref{thm:A}.
\end{proof}

\section*{Acknowledgments}
The authors would like to thank P.~Lorenzoni for useful discussions and remarks, and Y.~Zhang for a careful and detailed reading of the manuscript. This work was supported by the Netherlands Organization for Scientific Research.

\appendix

\end{document}